\documentclass{amsart}
\usepackage{amsmath,amsthm,amssymb,latexsym}
\usepackage{mathrsfs}
\usepackage{amsfonts}
\usepackage{amscd}
\usepackage{euscript}
\usepackage{graphicx}
\usepackage[all]{xy}
\usepackage[colorlinks=true, pdfstartview=FitV, linkcolor=blue,
citecolor=blue,
urlcolor=blue,breaklinks=true]{hyperref}

\newcommand{\sm}{\smallskip}
\newcommand{\ms}{\medskip}

\newcommand{\ot}{\otimes}
\def\gr*{^{{\rm gr}*}}

\def\co{\colon}
\newcommand{\tsum}{\textstyle \sum}
\newcommand\tl{\tilde}

\let\goth\mathfrak
\def\gg{\goth g} \newcommand{\g}{\gg}

\def\fra{\goth a}

\def\frh{\mathfrak{h}}
\def\fg{\mathfrak{g}}
\def\fp{\mathfrak{p}}

\def\scD{\mathscr{D}}
\def\scC{\mathscr{C}}
\def\scG{\mathscr{G}}
\def\scL{\mathscr{L}}
\def\caL{\mathcal L}
\def\scIBF{\mathbf{IBF}}
\def\scu{\mathsf{uni}}
\newcommand\lsl{\ensuremath{\mathfrak{sl}}}

\newcommand\sfx{\mathsf{x}}
\newcommand\sfy{\mathsf{y}}
\newcommand\sfz{\mathsf{z}}
\newcommand\sfibf{\mathsf{ibf}}
\newcommand\ibf{\sfibf}
\newcommand\ac{\mathsf{ac}}

\newcommand\al{\alpha}
\newcommand\be{\beta}
\newcommand\ga{\gamma}

\newcommand\de{\delta} \newcommand\De{\Delta}

\newcommand\ka{\kappa}
\newcommand\la{\lambda} \newcommand\La{\Lambda}
 \newcommand\vphi{\varphi}

\newcommand\si{\sigma} 

\newcommand\ta{\tau}

\newcommand\ze{\zeta}
\newcommand\om{\omega} 

\newcommand\bfA{\mathbf{A}}
\newcommand\bfM{\mathbf{M}}

\newcommand{\kalg}{k\mathchar45\mathbf{alg}}
\newcommand{\kALG}{k\mathchar45\mathbf{ALG}}

\newcommand{\Ralg}{R\mathchar45\mathbf{alg}}

\newcommand{\kMOD}{k\mathchar45\mathbf{MOD}}
\newcommand{\RMOD}{R\mathchar45\mathbf{MOD}}
\newcommand{\Rmod}{R\mathchar45\mathbf{mod}}
\newcommand{\kmod}{k\mathchar45\mathbf{mod}}

\newcommand\QQ{\mathbb{Q}}

\newcommand\ZZ{\mathbb{Z}}
\newcommand\AC{\mathbb{AC}}

 \DeclareMathOperator{\ad}{ad}

\DeclareMathOperator{\Aut}{Aut}
\DeclareMathOperator{\bfAut}{\mathbf{Aut}} 

 \DeclareMathOperator{\Cent}{Ctd}
 
 \DeclareMathOperator{\cent}{\Cent}

 \DeclareMathOperator{\GL}{GL}
 \DeclareMathOperator{\bfGL}{\mathbf{GL}} 
 \DeclareMathOperator{\Hom}{Hom}

 \DeclareMathOperator{\Id}{Id}

 \DeclareMathOperator{\IBF}{IBF}

 \DeclareMathOperator{\Span}{Span}
 \DeclareMathOperator{\Spec}{Spec}

 \DeclareMathOperator{\tr}{tr}

\DeclareMathOperator{\Zor}{Zor}
\DeclareMathOperator{\rmZ}{Z}

\theoremstyle{plain}
\newtheorem{theorem}{Theorem}[section]
\newtheorem{lemma}[theorem]{Lemma}
\newtheorem{lem}[theorem]{Lemma}
\newtheorem{prop}[theorem]{Proposition}
\newtheorem{cor}[theorem]{Corollary}
\newtheorem{princ}[theorem]{IBF-Principle}

\theoremstyle{definition}

\newtheorem{example}[theorem]{Example}

\newtheorem{con}[theorem]{Question}

\newtheorem{rem}[theorem]{Remark}

\newtheorem{definition}[theorem]{Definition}
\newtheorem{defn}[theorem]{Definition}
\newtheorem{descent}[theorem]{Descent Theory}
\newtheorem{invariance}[theorem]{Automorphism invariance}
\newtheorem{bas-ch}[theorem]{Base change}

\numberwithin{equation}{section} \allowdisplaybreaks

\begin{document}

\title{ Invariant bilinear forms of algebras given by faithfully flat descent}

\date{\today}
\author{E.~Neher}
\address{Department of Mathematics and Statistics, University of Ottawa,
Ottawa, Ontario K1N 6N5, Canada}
\thanks{E. Neher wishes to thank NSERC for partial support through a Discovery grant}\email{neher@uottawa.ca}
\author{A. Pianzola}
\address{Department of Mathematics, University of Alberta,
    Edmonton, Alberta T6G 2G1, Canada, and Centro  de Altos Estudios en Ciencias Exactas, Avenida de Mayo 866, (1084) Buenos Aires, Argentina.}
\thanks{A. Pianzola wishes to thank NSERC and CONICET for their continuous support}\email{a.pianzola@gmail.com}

\author{D. Prelat}
\address{Centro de Altos Estudios en Ciencias Exactas, Avenida de Mayo 866, (1084) Buenos Aires, Argentina.}
\thanks{D. Prelat is supported by a Research Grant from Universidad CAECE}
\email{dprelat@caece.edu.ar}

\author{C. Sepp}
\address{Centro de Altos Estudios en Ciencias Exactas, Avenida de Mayo 866, (1084) Buenos Aires, Argentina.}
\thanks{C. Sepp is supported by CONICET and Universidad CAECE}
\email{csepp@caece.edu.ar}

\keywords{Invariant bilinear form, Twisted form, Centroid, Faithfully flat descent.}
 \subjclass[2010]{Primary 17B67; Secondary 17B01, 12G05, 20G10.}

\maketitle

\begin{abstract}
 The existence of nondegenerate invariant bilinear forms is one of the most important tools in the study of Kac-Moody Lie algebras and extended affine Lie algebras. In practice, these forms are created, or shown to exist, either by assumption or in an ad hoc basis. The purpose of this work is to describe the nature of the space of invariant bilinear forms of certain algebras given by faithfully flat descent (which includes the affine Kac-Moody Lie algebras,  as well as Azumaya algebras and multiloop algebras) within a functorial framework. This will allow us to conclude the existence, uniqueness and nature of invariant bilinear forms for many important classes of algebras.
 \end{abstract}

 \pagestyle{myheadings}
 \markboth{ {\tiny Version \today}} {Version \today }

\section{Introduction}

One of the key ingredients for the study of Kac-Moody Lie algebras is the (generalized) Casimir operator. The existence of this operator is in turn based upon the existence of a symmetric invariant nondegenerate bilinear form on the Lie algebra.  Kac-Moody Lie algebras admitting such bilinear forms are called symmetrizable, and the most important example of these is given by the affine Kac-Moody Lie algebras. \sm

It is also known \cite{P2} that the affine algebras are precisely the twisted forms (given by Galois, hence also faithfully flat descent) of the ``split" affine algebras, namely of algebras of the form $\mathfrak{g} \otimes_k k[t^{\pm{1}}]$ where $\mathfrak{g}$ is a finite dimensional simple Lie algebra over an algebraically closed field $k$ of characteristic $0$. In this paper we shall establish the representability (in a functorial sense) and explicit description of the space of invariant $k$-bilinear forms for a large class of (in general infinite dimensional) algebras given by faithfully flat descent.

The techniques developed in this paper not only apply to Lie algebras, but to other classes of algebras as well, such as Azumaya algebras,  octonion, alternative or Jordan algebras; see \S \ref{applisec2} for an (incomplete) list of examples. Just as in the theory of Kac-Moody Lie algebras, the existence of a nondegenerate or even nonsingular invariant bilinear form on these algebras has important structural consequences. We therefore follow the approach of \cite{NP} and study invariant bilinear forms of arbitrary (nonassociative) algebras.

Among other things we will recover \cite[Lemma~2.3]{MSZ}, which considers this question for Lie algebras in the untwisted case. The need to consider graded invariant forms in the study of extended affine Lie algebras will require a ``graded version" of our main results. This will be given towards the end of the paper in \S\ref{sec:graded}, where we will provide, among other things, a classification-free proof of Yoshii's Theorem \cite{y:lie} for multiloop Lie tori stating that graded invariant bilinear forms are unique up to scalars.
We will also obtain new (as well as shed new light on known) results on Azumaya algebras (Theorem~\ref{teo:Azu}).

The technique that we use to describe the nature of invariant bilinear forms of algebras is based on two crucial ingredients:
\begin{itemize}
  \item That invariant bilinear forms are functorial in nature and that this functor is representable.

   \item Descent theory.
\end{itemize}
This leads us to outline a general theory of descent within a functorial setting which we find to be of independent interest. It is developed in \S\ref{sec:app} and later applied to the functor $\scIBF$ representing invariant bilinear forms. \sm

{\em Acknowledgement.} We thank Ottmar Loos for his many useful suggestions and comments. The first two authors gratefully acknowledge the hospitality of BIRS (Banff) and the Fields Institute (Toronto) where part of the research for this paper was carried out.
\bigskip

\section{Descent for functors stable under base change}\label{sec:app}

The goal of this section is to show that functors from the category of algebras to the category of modules that are stable under base change preserve descended forms. Relevant to our paper is the special case of the ``functor of invariant bilinear forms" $\scIBF$ as stated in Corollary~\ref{thm-ffdes}. To write the proof in this particular case, however,  obscures the more general nature of the construction that is taking place. We have thus chosen a framework that allows the essence of the argument to come across. Our setting, while somewhat abstract, is sufficiently ample for our purposes and of independent interest. By appealing to fibered categories, an even more general set up is possible in which the concept of functors stable under base change can be defined.
It is nevertheless a delicate question to identify, within this more general setting,  which arrows play the (crucial) role of faithfully flat base change and their accompanying descent theory. We leave it to the interested reader to explore such more general scenarios. \ms

\begin{definition}
[Categories $\kalg$, $\kALG$, $\kMOD$] \label{def:cate}
We fix once and for all a commutative associative unital ring $k$.

We denote by $\kalg$ the category of commutative associative unital $k$-algebras with unital $k$-linear algebra homomorphisms as morphisms. The symbol $R\in \kalg$ means that $R$ is an object in $\kalg$. We also use the notation $R/k$ to describe this situation. By definition, $R$ comes accompanied with a ``structure" ring homomorphism $\sigma_{R,k} \co k \to R$ under which $R$ is viewed as a $k$-module.

Of course since $R$ is a commutative ring, we also have the category $\Ralg$.  An object $S \in \Ralg$ will be viewed as an object of $\kalg$ via $\sigma_{S,k} = \sigma_{S,R} \circ \sigma_{R,k}$. The arrows of $\Ralg$ are then arrows of $\kalg$ and we have a natural forgetful functor $\Ralg \to \kalg$.

Let $ \alpha \co R \to S$ be a morphism in $\kalg$, $M$ be an $R$-module and $N$ an $S$-module. We say that a map $ f \co M \to N$ is an $\alpha${\it -semilinear module homomorphism} if it is additive and satisfies $f(rm) = \alpha(r)f(m)$ for all $r \in R$ and $m \in M$. Such a map $f$ is necessarily $k$-linear, namely if $M$ and $N$ are viewed as $k$-modules by means of the structure maps $\sigma_{R,k}$ and $\sigma_{S,k}$ then $f(cm) = cf(m)$ for all $c \in k$ and $m \in M$. If $M$ and $N$ have algebra structures and $f$ preserves multiplication, then we say that $f$ is an $\alpha${\it -semilinear algebra homomorphism}.

Given an $\alpha$-semilinear module homomorphism $ f \co M \to N$ as above and a morphism $\beta \co S \to T$ in $\kalg$, there exists a unique $\beta$-semilinear module homomorphism $f \ot \beta \co M \ot_R S \to N \ot_S T$ satisfying $(f \ot \beta) (m \ot s) = f(m) \ot \beta(s)$. It is clear that if $f$ is an algebra homomorphism, then so is $f \ot \beta$.

The objects of the category $\kALG$ are pairs $(R,A)$ consisting of
an $R\in \kalg$ and an $R$-algebra $A$. By this we mean
an $R$-module $A$ together with an $R$-bilinear map $A \times A \to A$, $(a_1, a_2) \mapsto a_1 a_2$. In particular, we do not require any further identities (even though Lie algebras are our main interest, all algebras are considered). A  {\it morphism} in $\kALG$, written as $(\al, f) \co (R,A) \to (S,B)$, is a pair consisting of a morphism $\al \co R \to S$ in $\kalg$ together with an $\al$-semilinear algebra homomorphisms $f\co A \to B$.\footnote{ The category $\kALG$ appears in similar form in \cite[3.2]{LoPeRa}.}

The category $\kMOD$ is defined in complete analogy to $\kALG$, with algebras replaced by modules and algebra homomorphisms replaced by module homomorphisms. Thus, an object in $\kMOD$ is a pair $(R, M)$ consisting of some $R\in \kalg$ and an $R$-module $M$ and  a morphism in $\kMOD$ is a pair $(\al, f) \co (R, M) \to (S, N)$ consisting of a morphism $\al \co R \to S$ in $\kalg$ and an $\al$-semilinear module homomorphism $f\co M \to N$. \sm

\end{definition}

\begin{definition}[Base change]\label{def:base-change}
Both categories $\kALG$ and $\kMOD$ admit base change by objects of $\kalg$. We explain this for $\kALG$ and at the same time also set up our notation. Let $\al \co R \to S$ be a morphism in $\kalg$ and let $A$ be an $R$-algebra. Viewing $S$ as an $R$-algebra via $\al$, the tensor product $A\ot_R S$ is an $S$-algebra whose product is given by $(a_1 \ot s_1) (a_2 \ot s_2) = (a_1 a_2) \ot (s_1 s_2)$ for $a_i \in A$ and $s_i \in S$. Since we will in this section repeatedly use different algebra homomorphisms between $R$ and $S$, it will be useful to temporarily employ the more precise notation $A\ot_\al S$ instead of the traditional $A\ot_R S$. We will also appeal to this notation elsewhere in the paper whenever it is necessary to emphasize the structure map $\al$.
We define \[ \al^A \co A \to A\ot_\al S, \quad a \mapsto a\ot 1_S,
\]
and note that $(\al, \al^A) \co (R,A) \to (S, A\ot_\al S)$ is a morphism in $\kALG$. Moreover, for any morphism $(\al, f) \co (R, A) \to (S, B)$ in $\kALG$ there exists a unique $S$-linear algebra homomorphism $f^\alpha\co {A \ot_\al S} \to B$ satifying $f^\alpha(a\ot s) = sf(a)$. It is clear from the definitions that
\begin{equation}\label{def:base-change1}
\vcenter{
 \xymatrix{(R, A) \ar[rr]^{(\al, f)} \ar[dr]_{(\al, \al^A)}
   && (S,B) \\  & (S, A\ot_\al S) \ar[ur]_{(\Id_S, f^\alpha)}
}}\end{equation}
is a commutative diagram in $\kALG$. Assume that $\be \co S \to T$ is another morphism in $\kalg$. We can then add to the diagram~\eqref{def:base-change1} the morphism $(\be, \be^B) \co (S, B) \to (T, B\ot_\be T)$ and obtain the diagram
\begin{equation}\label{def:base-change2}
\vcenter{
\xymatrix{(R, A) \ar[r]^{(\al, f)} \ar[d]_{(\al, \al^A)}
   & (S,B)  \ar[d]^{(\be, \be^B)}  \\
    (S, A\ot_\al S) \ar[r]_{(\be, f\ot \be)} &(T, B \ot_\be T)
   } }
\end{equation}
which commutes since  $(\be, f\ot\be) = (\be,\be^B) \circ (\Id_S, f^\alpha)$. One should view this diagram as the base change of $(\al,f)$ by $\beta$. \sm
\end{definition}

\begin{definition}[Functors over $\kalg$]\label{def:fun-fib}
The projection onto the first component defines functors $\Pi_\bfA \co \kALG \to \kalg$ and $\Pi_\bfM \co \kMOD \to \kalg$. We say that a functor $F \co \kALG \to \kMOD$ \textit{is a functor over} $\kalg$ if  $\Pi_\bfM \circ F = \Pi_\bfA$, i.e.,
\[  \xymatrix{\kALG \ar[rr]^F \ar[dr]_{\Pi_\bfA} &&\kMOD \ar[dl]^{\Pi_\bfM} \\ & \kalg } \]
is a commutative diagram. Such a functor $F$ maps an object $(R,A)\in \kALG$ to $F(R,A) = (R, F_R(A))$ for some $R$-module $F_R(A)$, and it sends a morphism $(\al, f) \co (R,A) \to (S, B)$ to $F(\al,f) = (\al, F_\al(f))$ for some $\al$-semilinear map $F_\al(f) \co F_R(A) \to F_S(B)$.

\noindent {\bf Convention.} In what follows and without further explication, if $R=S$ and $\al = \Id_S$ then we will simply write $F(f)$ instead of $F_\al(f)$.

 Given a morphism $\al \co R \to S$ in $\kalg$ and $(R,A) \in \kALG,$ we can apply $F$ to the morphism $(\al, \al^A) \co (R,A) \to (S, A\ot_\al S)$ in $\kALG$ of \ref{def:base-change} and get a morphism in $\kMOD$
\[ F(\al, \al^A) =(\al, F_\al(\al^A)) \co (R, F_R(A)) \to (S, F_S(A\ot_\al S)).\]
Since this map is $\alpha$-semilinear, it induces an $S$-linear map
\begin{equation} \label{def:fun-fib1}
 \nu_{A, \al}^F \co F_R(A)\ot_\al S \to F_S(A\ot_\al S), \quad
     m\ot s \mapsto s F_\al(\al^A)(m)
   \end{equation}
of $S$-modules. These maps will play an essential role in our work for they constitute the essential ingredient in the definition of functors stable under base change. For convenience in what follows, if $F$ is fixed in the discussion, we will denote $ \nu_{A, \al}^F$ simply by $ \nu_{A, \al}$.
\end{definition}

\begin{lem}\label{lem:tran}
Suppose that $F \co \kALG \to \kMOD$ is a functor over $\kalg$. \sm

{\rm (a) ($F$ and $\nu$ commute)} Let $(\al, f)  \co (R,A) \to (S,B)$ be a morphism in $\kALG$ and let $\be \co S \to T$ be a morphism in $\kalg$. Then the diagram
\[  \xymatrix@C=50pt{
    F_R(A) \ot_\al S  \ar[r]^{F_\al(f) \ot \be} \ar[d]_{\nu_{A,\al}}
     & F_S(B) \ot_\be T \ar[d]^{\nu_{B, \be}} \\
    F_S(A\ot_\al S) \ar[r]_{F_\be(f\ot \be)} & F_T(B \ot_\be T) }
  \]
commutes.

{\rm (b) (Transitivity)} Let $R \xrightarrow{\al} S \xrightarrow{\be} T$ be morphisms in $\kalg$ and let $(R,A) \in \kALG$. Then the diagram
\[  \xymatrix@C=50pt{
    F_R(A) \ot_\al S  \ar[r]^{\Id_{F_R(A)} \ot \be} \ar[d]_{\nu_{A,\al}}
     & F_R(A) \ot_{\be\circ \al} T \ar[d]^{\nu_{A, \be \circ \al }} \\
    F_S(A\ot_\al S) \ar[r]_{F_\be(\Id_A \ot \be)} & F_T(A \ot_{\be\circ \al} T) }
  \]commutes.
\end{lem}

\begin{proof} (a) The  maps $f \ot \beta$ and  $ F_\al(f) \ot \be$  are described  in Definitions \ref{def:cate} and \ref{def:base-change}.  The bottom horizontal arrow is obtained by applying $F$ to $(\beta, f \ot \beta)$.
For $m\in F_R(A)$ and $s\in S$ we have
\begin{align*}
  \big( \nu_{B, \be} \circ (F_\al(f) \ot \be )\big)(m\ot s) &=
     \be(s) F_\be(\be^B) \big(F_\al(f)(m)\big), \, \text{\rm while}\\
  \big(F_\be(f\ot \be) \circ \nu_{A, \al} \big) (m\ot s) &=
      \be(s) F_\be(f \ot \be) \big(F_\al(\al^A)(m)\big).
  \end{align*}
It therefore suffices to show that $F_\be(\be^B) \circ F_\al(f)= F_\be(f \ot \be) \circ F_\al(\al^A)$. But this follows by applying the functor $F$ to the commutative diagram~\eqref{def:base-change2}. \sm

(b) With the notation of (a) we have
\begin{align*}
  \big(\nu_{A, \be \circ \al}\circ (\Id \ot \be) \big)(m\ot s) &= \be(s) \, F_{\be \circ \al}\big( (\be \circ \al)^A\big) (m), \, \text{\rm while} \\
  (F_\be(\Id_A \ot \be) \circ \nu_{A, \al})\, (m\ot s)
  &= \be(s)F_\be(\Id_A \ot \be) \, F_\al(\al^A)\, (m).
  \end{align*}
It is therefore sufficient to show that $F\big( (\be \circ \al)^A\big) = F_\be( \Id_A \ot \be) \circ F_\al( \al^A)$. By functoriality, this is a consequence of $\big((\be \circ \al), (\be \circ \al)^A\big) = (\be, \Id_A \ot \be) \circ (\al, \al^A)$. \end{proof}

\begin{definition}[Functors stable under base change]\label{fun-comm}
Let $F\co \kALG \to \kMOD$ be a functor over $\kalg$. We will say {\em $F$ is stable under base change}  if for all morphisms $\al\in \kalg$ and all $(R,A)\in \kALG$ the $S$-module homomorphism $\nu_{A, \al}^F \co F_R(A) \ot_\al S \to F_S(A \ot_\al S)$ defined in (\ref{def:fun-fib1}) is an isomorphism.\end{definition}

\begin{example}\label{derived} An example of a functor stable under base change is the ``invariant bilinear form functor'' $\scIBF$ of \ref{def:repr}, see Proposition~\ref{ibf-fun}(b), which is most relevant to our work. Of course the quintessential example of a functor $F \co  \kALG \to \kMOD$ over $\kalg$ stable under base change is the ``tensor product" functor that attaches to $(R,A)$ the  $R$-module $A \ot_R A $  with the natural definition of $F$ at the level of arrows. The stability of base change is given by the canonical $S$-module isomorphism $(A \ot_R A)\ot_R S \simeq (A \ot_R S) \ot_S (A \ot_R S)$.

To justify the generality of this section we give another example. Let $A$ be an $R$-algebra. Recall that its derived algebra is defined by
\[ \scD(A) = \Span_R\{ a_1 a_2 : a_i \in A\} = \Span_\ZZ\{ a_1 a_2 : a_i \in A\}.\]
If $(\al,f)\co (R,A) \to (S,B)$ is a morphism in $\kALG$, then $f\big(\scD(A)\big) \subset \scD(B)$. Hence by restriction we obtain a map $\scD_\al(f) \co \scD(A) \to \scD(B)$. If we view $\scD(A)$ as a submodule of the $R$-module $A$ it is immediate that $\big(\al, \scD_\al(f)\big)\co \big(R, \scD(A)\big) \to \big( S, \scD(B)\big)$ is a morphism in $\kMOD$ and that the assignment
\[ (R,A) \mapsto \big(R, \scD(A)\big)\quad \hbox{and} \quad (\al,f) \mapsto \big( \al, \scD_\al(f)\big) \]
defines a functor $\scD\co \kALG \to \kMOD$ over $\kalg$.

The explicit nature of the map $\nu_{A, \al}^\scD \co  \scD (A) \ot_\al S \to \scD (A \ot_\al S)$ is clear: $a_1a_2 \ot s \mapsto (a_1 \ot s)(a_2 \ot 1) = a_1a_2 \ot s = (a_1 \ot 1)(a_2 \ot s)$. Thus $\nu_{A, \al}^\scD$ is always surjective. This map, however, need to be injective since there is no reason for the natural map $\scD (A) \ot_\al S \to A \ot_\al S$ to be so (it is, for example, if $\al: R \to S$ is flat).

Consider a new functor $F \co  \kALG \to \kMOD$ over $\kalg$ which assigns to $(R,A)$ the pair $\big(R , A/ \scD(A)\big)$ and is defined at the level of arrows in the natural way. We leave it to the reader to check that the surjectivity of $\nu_{A, \al}^\scD$ easily implies that $\nu_{A, \al}^F$ is an isomorphism of $S$-modules. Thus $F$ is stable under base change.

\smallskip
\end{example}

\begin{rem}\label{rescate}  {\rm The results of this section can be generalized by replacing $\kalg$, $\kALG$ and $\kMOD$ by subcategories stable under base change and by modifying the Definition~\ref{def:fun-fib} correspondingly. As we shall see, the most relevant case for us is that of functors which are stable under faithfully flat base change. We will leave it to the interested reader to work out the necessary axioms. }
\end{rem}

\begin{descent}[Faithfully flat descent of modules and algebras]\label{descent} We give a short review of the descent theory of modules and algebras. Our ultimate objective is to outline a descent theory in the setting of functors stable under faithfully flat base change. We also use the opportunity to introduce notation and a presentation of descent theory that is implicitly, but not explicitly used in the standard references (\cite{KO, SGA1, Wht}), cf.\ \cite{P}. Without these the formalism for descent in the functorial setting is impossible to redact.
\sm

Assume that $S/R$ is faithfully flat. We let $ S''=S\ot_R S$ and denote by $\al_i:S\to S''$, $i=1,2$, the ``projections''  defined  by
\[
 \al_1(s)=s\ot 1 \quad\text{\rm and} \quad \al_2(s)=1\ot s,\]
which allow us to view $S''$ as an $S$-algebra in two different ways. Note that since $\al \co  R \to S$ is faithfully flat, $\al_1 \circ \al = \al_2 \circ \al$. Suppose $M$ and $N$ are $R$-modules and that $N$ is an $S/R$-form of $M$. Thus there exists an $S$-module isomorphism $\theta \co (M\ot_\al S) \to (N \ot_\al S)$. To $\theta$ and $i=1,2$ we associate the $S''$-module isomorphisms $\theta_i$ defined by the following commutative diagram.
\begin{equation} \label{def:thet1}  \vcenter{
\xymatrix@C=50pt{
    (M \ot_\al S) \ot_{\al_i} S'' \ar[r]^{\theta \ot_{\al_i} \Id_{S''}}_\simeq
       \ar[d]_{\ta_i^M}^\simeq
          & (N \ot_\al S) \ot_{\al_i} S''  \ar[d]^{\ta_i^N}_\simeq\\
       M \ot_{\al_i \circ \al } S'' \ar[r]^{\theta_i}_\simeq & N \ot_{\al_i \circ \al} S''
} } \end{equation}
Here $\ta_i \co S\ot_{\al_i} S'' \to S''$ is defined by $\ta_i(s_1 \ot s_2 \ot s_3) = \al_i(s_1) (s_2 \ot s_3)$, e.g.\ $\ta_2(s_1 \ot s_2  \ot s_3) = (s_2  \ot s_1s_3)$, while $\ta_i^M \big((m \ot s_1) \ot s_2 \ot s_3\big) = m \ot \al_i(s_1)(s_2 \ot s_3)$. The maps $\ta_i^N$ are defined similarly. In what follows, $\ta_i^M$ and $\ta_i^N$ are viewed as $S''$-linear maps.

The situation can be summarized by the following commutative diagram
\begin{equation} \label{dia:desc} \vcenter{
\xymatrix{ 0 \ar[r] &
   M \simeq M\ot_{\Id_R}R\ar[rr]^{\Id_M \ot \al} && M\ot_\al S \ar@<0.5ex>[rr]^{\Id_M \ot \al_1} \ar@<-0.5ex>[rr]_{\Id_M \ot\al_2} \ar[d]_\simeq^\theta
   && M \ot_{\al_i \circ \al} S''
   \ar@<0.5ex>[d]^>>>>>{\theta_1} \ar@<-0.5ex>[d]_>>>>>{\theta_2}
  \\ 0 \ar[r]&
  N\simeq N\ot_{\Id_R}R \ar[rr]^{\Id_N \ot \al} && N \ot_\al S \ar@<0.5ex>[rr]^{\Id_N \ot \al_1} \ar@<-0.5ex>[rr]_{\Id_N \ot \al_2}
   && N \ot_{\al_i \circ \al} S''
    }  }
  \end{equation}
The  rows are exact since $\al \co  R \to S$ is faithfully flat, e.g., $\Id_M \ot \al$ is injective  and its image  $M \ot 1_S$ is the $R$-submodule $M \ot_\al S$ where $\Id_M \ot \al_1$ and $\Id_M \ot \al_2$ agree. The $S/R$-cocycle $u$ defining the $S/R$-form $N$ is \[ u = \theta_2^{-1} \circ \theta_1 \in \Aut_{S''}(M\ot_{\al_1 \circ  \al} S'') = \Aut_{S''}(M\ot_R S'')\] as we now explain.\footnote{Note that $u$ can indeed be viewed as an $S''$-module automorphism $M\ot_{\al_1 \circ  \al} S''$  because $\al_1 \circ  \al = \al_2 \circ  \al$. In what follows we will view $S''$ as an $R$-algebra via either one of these two (equal) maps. The notation $M \ot_R S''$ responds to this convention.}
Let
\begin{equation}\label{desform}
L = \{ x \in M\ot_\al S : u\big((\Id_M \ot \al_1)(x)\big) = (\Id_M \ot \al_2)(x)\}.
\end{equation}
It is clear that $L$ is an $R$-submodule of $M \ot_\al S = M \ot_R S$,  and a simple diagram chase in (\ref{dia:desc}) above shows that the restriction of $\theta$ to $L$ induces an isomorphism with $N \ot 1_S \subset N \ot_\al S= N \ot_R S$. In other words, up to $R$-module isomorphism, our module $N$ corresponds to the cocycle $u$.

It is well-known (and easy to check in any case) that $u$ is a cocycle, i.e. that $u \ot \al_{1,3} = (u \ot \al_{2,3}) \circ (u \ot \al_{1,2})$ where the $\al_{i,j} \co S'' \to S''' = S \ot_\al S \ot_\al S$ are the natural $S''$-algebra morphisms defined by putting $1_S$ in the position $l\ne i,j$. The cocycle condition can be rewritten in the form $\theta_{1,3} = \theta_{2,3} \circ \theta_{1,2}$ where the $\theta_{i,j}$ are automorphisms of the $S'''$-module $M \ot_{\al_{i,j}\circ \al_i \circ \al} S'''= M \ot_R S'''$ defined using a diagram similar to (\ref{def:thet1}).

For $R$-algebras the situation is identical. Say that both $A$ and $B$ are $R$-algebras and that our isomorphism $\theta$ above is now an $S$-algebra isomorphism. Then $u$ is an $S''$-algebra automorphism of $A\ot_R S''$,  the descended $R$-module $L$ is an $R$-subalgebra of $A \ot_\al S=A \ot_R S$ and  the restriction of $\theta$ induces an $R$-algebra isomorphism between $L$ and $B \simeq B \otimes 1_S$.\end{descent}

\begin{theorem}\label{mainfunctor} Let $A$ be an $R$-algebra, and let $B$ be an $S/R$ form of $A$ determined by the cocycle $u$ as described in {\rm \ref{descent}} above. If $F\co \kALG \to \kMOD$ is a functor over $\kalg$ stable under base change, then $F_R(B)$ is an $S/R$-form of the $R$-module $F_R(A)$ which is isomorphic as an $R$-module to the one given by the cocycle \[\nu_{A, \al_2 \circ \al}^{-1} \circ F(u) \circ \nu_{A, \al_1 \circ \al} \in {\rm Aut}_{S''}\big(F_R(A) \ot_R  S''\big).\]
\end{theorem}

\begin{proof} We fix an $S$-algebra isomorphism $\theta \co  A \ot_{\al}S \to B \ot_{\al} S$.  The cocycle determining $B$ (up to $R$-algebra isomorphism) is then $u=\theta_2^{-1} \circ \theta_1 \in \Aut_{S''}(A \ot_R S'')$. The result to establish can thus be rephrased by saying that
\begin{enumerate}
\item[(a)]
$ z:= \big( \nu_{A, \al_2 \circ \al}^{-1} \circ F(\theta_2)^{-1} \circ \nu_{B, \al_2 \circ \al} \big) \, \circ \, \big( \nu_{B,\al_1 \circ \al}^{-1} \circ F(\theta_1) \circ \nu_{A, \al_1\circ \al}\big) \in {\rm Aut}_{S''}\big(F_R(A) \ot_R  S''\big)$ is a cocycle.

\item[(b)]  The $R$-module determined by the cocycle $z$ is isomorphic to $F_R(B)$. \end{enumerate}
Let us define $F^\nu(\theta) = \nu_{B,\al}^{-1} \circ F(\theta) \circ \nu_{A, \al}$ by means of the diagram
\[ \xymatrix@C=50pt{ F_R(A) \ot_\al S \ar[r]^{F^\nu(\theta)}
     \ar[d]_{\nu_{A, \al}} & F_R(B) \ot_\al S \\
   F_S(A\ot_\al S) \ar[r]^{F(\theta)} & F_S(B \ot_\al S)
       \ar[u]_{(\nu_{B,\al})^{-1}}
 } \]
In view of the descent of modules construction explained in  \ref{descent}, to prove (a) and (b) it will suffice to show that
\begin{equation}\label{*}
\nu_{B, \al_i \circ \al}^{-1} \circ F(\theta_i) \circ \nu_{A, \al_i \circ \al} = \big(F^\nu(\theta)\big)_i \quad \hbox{for $i=1,2$.}\end{equation}
Both cases $i=1,2$ are similar and we check in detail the case $i =1$ only. We will use the following commutative diagram of $S''$-module isomorphisms.
\[ \xymatrix@C=60pt@R=30pt{
  F_R(A) \ot_{\al_1 \circ  \al } S''  \ar[r]^{\big(F^\nu(\theta)\big)_1}
        \ar[d]_{(    \Id_{F_R(A)} \ot \ta_1)^{-1}} \ar@/_6pc/@{.>}[dddd]_{\nu_{A,\al_1\circ \al}}
                 &    F_R(B) \ot_{\al_1 \circ  \al} S''
                    \ar[d]^{( \Id_{F_R(B)} \ot \ta_1)^{-1}}
           \ar@/^6pc/@{.>}[dddd]^{\nu_{B,\al_1\circ \al}}
     \\
    F_R(A) \ot _\al S \ot_{\al_1} S'' \ar[r]^{ F^\nu(\theta) \ot \Id_{S''}}
       \ar[d]_{\nu_{A,\al} \ot \Id_{S''}}
        &  F_R(B) \ot_\al S\ot_{\al_1} S'' \ar[d]^{\nu_{B,\al} \ot \Id_{S''}}
     \\
    F_S(A \ot_\al S) \ot_{\al_1} S'' \ar[r]^{ F(\theta) \ot \Id_{S''}}
       \ar[d]_{\nu_{A\ot_\al S, \al_1}}
         & F_S(B\ot_\al S) \ot_{\al_1} S''   \ar[d]^{\nu_{B \ot_\al S, \al_1}}
      \\
     F_{S'' }\big( (A \ot_\al S) \ot_{\al_1} S''\big)
       \ar[r]^{F( \theta \ot_{\al_1} \Id_{S''})} \ar[d]_{F(\ta_1^A)}
       & F_{S''}\big((B \ot_\al S) \ot_{\al_1} S''\big)
          \ar[d]^{F( \ta_1^B) }
      \\
      F_{S''}(A \ot_{\al_1\circ \al} S) \ar[r]^{F(\theta_1)}
       & F_{S''}(B \ot_{\al_1 \circ \al} S'')
        } \]
The top rectangle commutes by definition of $F^\nu(\theta_1)$; the second rectangle commutes  by applying the base change $\al_1\co  S \to S''$ to the diagram defining $F^\nu(\theta)$; the third rectangle commutes by Lemma~\ref{lem:tran}(a) for $(R,A), (S,B), (\al , f)$ and $\be$ replaced by $(S, A\ot_\al S), (S, B\ot_\al S), (\Id_S, \theta)$ and $\Id_{S''}$ respectively; the bottom rectangle commutes by applying $F$ to the diagram~\eqref{def:thet1} defining $\theta_1$. For the proof of \eqref{*} it is therefore sufficient to show that the dotted maps equal $\nu_{A, \al_1 \circ \al}$ and $\nu_{B, \al_1\circ \al}$ respectively. We check the case of $\nu_{A, \al_1 \circ \al}$ by following explicitly the  arrows on the left of the diagram. The case of  $\nu_{B, \al_1\circ \al}$, which is similar,  is left to the reader.
  Let $m\in F_R(A)$ and $s_1, s_2 \in S$. Then
\begin{align*}
   m \ot s_1 \ot s_2 &\xrightarrow{(\Id_{F_R(A)} \ot \ta_1)^{-1}}
      m\ot s_1  \ot 1_S \ot s_2
     \xrightarrow{\nu_{A,\al} \ot \Id_{S''}}
   \big( s_1 F_\al(\al^A)(m) \big) \ot (1_S \ot s_2)
  \\ & \xrightarrow{\nu_{A\ot_\al S, \al_1}}
       (1_S \ot s_2)  \Big( F_{\al_1}(\al_1^{A\ot_{\al} S})
           \big(s_1 F_\al(\al^A)(m)\big) \Big)
 \\ & \qquad \qquad  = (1_S \ot s_2) \Big(  \al_1(s_1) F_{\al_1}(\al_1^{A\ot_{\al} S})
               \big( F_\al(\al^A)(m)\big)\Big)
    \\ & \qquad \qquad = (1_S \ot s_2) (s_1 \ot 1_S)\Big(  F_{\al_1}(\al_1^{A\ot_{\al} S}) \big( F_\al(\al^A)(m)\big)\Big)
  \\ & \qquad \qquad
     = (s_1 \ot s_2) \Big(  F_{\al_1}(\al_1^{A\ot_{\al} S})
               \big( F_\al(\al^A)(m)\big)\Big)
  \\ & \quad \xrightarrow{ F(\ta_1^A)}
      (s_1 \ot s_2) F(\ta_1^A)\Big(  F_{\al_1 }(\al_1^{A\ot_{\al} S})
               \big( F_\al(\al^A)(m)\big)\Big)
 \\& \qquad \qquad
    = (s_1 \ot s_2)\,  F_{\al_1 \circ \al}\big( (\al_1 \circ \al)^A(m)\big).
     \end{align*}
This completes the proof since by definition $ \nu_{A,\al_1\circ \al}(m \ot s_1 \ot s_2) = (s_1 \ot s_2)\,  F_{\al_1 \circ \al}\big( (\al_1 \circ \al)^A(m)\big)$. \end{proof}

\begin{rem}\label{over} In the case that we are most interested in, the  $R$-algebra $A$ is of the form $A =  \fra \ot_k R$ for some $k$-algebra $\fra$. The isomorphism $\theta$ can now be thought as an $S$-algebra isomorphism  (also denoted by $\theta$)
$$\theta \co  \fra \ot_k S \simeq (\fra \ot_k R) \ot_R S = A \ot_R S  \to B \ot_R S$$
where we have denoted $\ot_{\al}$ by $\ot_R$. Under the canonical $S''$-isomorphism $\fra \ot_k S'' \simeq (\fra \ot_k R) \ot_R S''$ (where we recall that $S''$ is viewed as an $R$-algebra via $\al_1 \circ \al = \al_2 \circ \al$) we can view the corresponding  cocycle $u$ as an $S''$-algebra automorphism of $\fra \ot_k S''$. The descended module $L$ is then given by
\begin{equation}\label{desformk}
L = \{ x \in \fra\ot_k S : u\big((\Id_\fra \ot \al_1)(x)\big) = (\Id_\fra \ot \al_2)(x)\}.
\end{equation}

For future reference we record the following ``descent setting" that covers the case that we are most interested in:
\begin{equation}\label{descent-setting}
   \begin{cases}
     \hbox{$\fra$ is an algebra  over $k$;}\\
     \hbox{$R\in \kalg$ is such that $R/k$ is flat;}\\
      \hbox{$S\in \Ralg$ is such that $S/R$ is faithfully flat;}\\
      \hbox{$B$ is an $R$-algebra which is an $S/R$-form of $A = \fra \ot_k R$.}\\
        \end{cases}
\end{equation}
\end{rem}


\section{Invariant functions}\label{invfunsec}

Unless stated otherwise, $k$ is a commutative associative unital ring, $R$ is a commutative associative unital $k$-algebra, namely $R\in \kalg$ in the notation of \S\ref{sec:app}, $M$ is an $R$-module, $V$ is a $k$-module and $B$ is an arbitrary (not necessarily associative, unital...) $R$-algebra, i.e., $(R,B) \in \kALG$.
Our goal is to study invariant bilinear functions $B\times B \to V$. We begin with pertinent definitions.

\begin{defn}[$(R,k)$-bilinear functions]\label{Rk-bil}
A $k$-bilinear function $\be\co M \times M \to V$ is called \emph{$(R,k)$-bilinear} if $\be(rm_1  , m_2) = \be(m_1, r m_2)$ holds for all $m_i \in M$ and $r\in R$. We denote by $\scL^2_{(R,k)}(M; V)$ the $R$-module of $(R,k)$-bilinear maps $M \times M \to V$. Its $R$-module structure is given by $(r\be)(m_1, m_2) = \be(rm_1, m_2)$ for $r\in R$ and $\be \in \scL^2_{(R,k)}(M; V)$. We abbreviate $\scL^2_R(M) = \scL^2_{(R,R)}(M;R)$. \sm

It is immediate from \cite[II, \S4.1]{bou:Aa} that one has a commutative triangle \begin{equation} \label{Rk-bil1}
   \vcenter{
\xymatrix{
      \Hom_k(M\ot_R M,V) \ar[rr]^{\sfx}_\simeq
             \ar[dr]^\simeq_{\sfy} && \scL^2_{(R,k)}(M;V)
                \\
            & \Hom_R\big(M, \Hom_k(M,V)\big) \ar[ur]_{\sfz}^\simeq
}}\end{equation}
of $R$-linear isomorphisms: For $\vphi \in \Hom_k(M \ot_R M, V)$ one defines $\big(\sfx(\vphi)\big)(m_1, m_2) = \vphi(m_1 \ot m_2)$. This is an $R$-linear map with respect to the $R$-action on $\Hom_k(M \ot_R M, V)$ given by $(r \vphi)(m_1\ot m_2) = \vphi(rm_1 \ot m_2)$ for $r\in R$. Also, $\big(\sfy(\vphi)\big)(m_1)$ maps $m_2$ onto $\vphi(m_1 \ot m_2)$. The map $\sfy$ is $R$-linear if we view $\Hom_k(M,V)$ as an  $R$-module via $(rh)(m)= h(rm)$ for $h\in \Hom_k(M,V)$. To $\be \in \Hom_R\big(M, \Hom_k(M,V)\big)$ one associates the $(R,k)$-bilinear function $\sfz(\be) $ defined by $\big(\sfz(\be)\big)(m_1, m_2) = \big(\be(m_1)\big)(m_2)$.
 \sm

One calls $\be \in \scL^2_{(R,k)}(M;k)$ \textit{nondegenerate} (respectively \textit{nonsingular}) if $\sfz^{-1}(\be) \in  \Hom_R(M,M^*)$, $M^* = \Hom_k(M,k)$, is injective (respectively bijective).\footnote{The asymmetry in these definitions (one should, strictly speaking, speak of left and right nondegeneracy and nonsingularity) will not play a major role in this paper since our main interest later will be in invariant bilinear forms of perfect algebras which, by Remark~\ref{rem:perfect}, are symmetric.} In more familiar terms $\be \in \scL^2_{(R,k)}(M;k)$ is nondegenerate if and only if $\be(b,M) = 0$ implies $b=0$. If $k$ is a field, the existence of a nonsingular bilinear form on a $k$-algebra $B$ forces $B$ to be of finite dimension. Moreover, for a finite-dimensional $B$ a form is nondegenerate if and only if it is nonsingular. See Lemma~\ref{ns-char} for a generalization.
\end{defn}

\begin{definition}[Invariant functions] \label{def:inb} We call $\beta \in \scL^2_{(R,k)}(B;V)$ \emph{invariant} if
\begin{equation} \label{def:inb1}
  \be(ab,c) = \be(a, bc) = \be(b,ca)
\end{equation}
holds for all $a,b,c\in B$.

Clearly, the set $\IBF_{(R,k)}(B;V)$ of all invariant $(R,k)$-bilinear functions $B\times B \to V$ is a submodule of the $R$-module $\scL^2_{(R,k)}(B;V)$. The following special cases of $\IBF_{(R,k)}(B; V)$ are of particular interest: \begin{align*}
  \IBF_k(B;V) :&= \IBF_{(k,k)}(B; V), \\
  \IBF_{(R,k)}(B) :&= \IBF_{(R,k)}(B;k),\\
    \IBF_k(B) :&= \IBF_k(B;k) = \IBF_{(k,k)}(B). \\
    \end{align*}
The elements of  $\IBF_k(B)$ are called \emph{invariant $k$-bilinear forms}. Of particular importance is the case $k=R$; these are the invariant $R$-bilinear forms on $B$.
\end{definition}

\begin{rem}
The above definition works equally well for invariant bilinear functions on dimodules of algebras (for the definition of a dimodule see \cite{NP} as well as Lemma~\ref{inv-cent} supra). This is not without interest as these types of bilinear forms are an important tool for the study of the representation theory of the algebras in question, for example for the existence of the Jantzen filtration of Verma modules. We will not pursue this more general set up in this work. \end{rem}

\begin{rem}\label{rem:perfect}
If $B$ is perfect, namely if $B=BB$, where $BB=\Span_R\{ ab: a,b,\in B\}$, every invariant bilinear function is symmetric: $\be(ab,c) = \be(a,bc) = \be(b,ca) = \be(bc,a) = \be(c,ab)$. Moreover, any $k$-bilinear function is already $(R,k)$-bilinear:
\[
   \IBF_k(B;V)  = \IBF_{(R,k)}(B;V) \qquad (\hbox{$B$ perfect}).
\]
Indeed, $\be\big(r(ab), c\big) = \be\big( a (rb), c\big) = \be\big(a, (rb)c\big) = \be\big( a, b(rc)\big) =\be\big( ab, rc\big)$.
\end{rem}

\begin{definition}[Universal invariant function] \label{def:repr}Let $\scIBF_R(B)$ be the quotient of the $R$-module $B \ot_R B$ by the submodule
\begin{align}
 \sfibf_R(B) &= \Span_R\{ ab \ot c - a \ot bc , \, ab\ot c - b \ot ca : a,b,c\in B\} \nonumber \\
    &= \Span_\ZZ \{ ab \ot c - a \ot bc , \, ab\ot c - b \ot ca : a,b,c\in B\} \nonumber \\
   & = \Span_\ZZ\{ ab \ot c - a \ot bc, \, a\ot bc - bc \ot a : a,b,c \in B\}  \label{def:repr-0},
 \end{align}
where the last equality follows from
$ (bc \ot a - a \ot bc) + (ab \ot c - b \ot ca)
     = (ab \ot c - a \ot bc) + (bc \ot a - b \ot ca).
$ Denoting by $i_B \co \sfibf_R(B) \to B\ot_R B$ the inclusion and by $q_B$ the canonical quotient map
\[
   q_B\co B \otimes_R B \to \scIBF_R(B), \quad
    a \otimes b \mapsto \overline{a \ot b},
\]
we  have an exact sequence of $R$-modules
\begin{equation}\label{def:rep-ex}
  0 \to \sfibf_R(B)\xrightarrow{i_B} B \ot_R B \xrightarrow{q_B} \scIBF_R(B)
   \to 0 .
\end{equation}
We define an invariant $R$-bilinear function
\begin{equation} \label{def:rep-ex2}
 \be_{\scu}\co B \times B \to \scIBF_R(B), \quad
    \be_\scu(a, b) =  \overline{a\ot b},
\end{equation}
called the \emph{universal invariant $R$-bilinear function}.\footnote{We want to thank K.-H.\ Neeb for bringing this concept to our attention.} This terminology is justified  because of the following natural $R$-module isomorphism
\begin{equation} \label{iso-hom-ibf-1}
 \Hom_k( \scIBF_R(B), V) \xrightarrow{\simeq} \IBF_{(R,k)}(B; V),  \quad     f \mapsto \tilde f =f \circ \beta_\scu.
\end{equation}
Its inverse
\begin{equation} \label{iso-hom-ibf-2}
\IBF_{(R,k)}(B; V) \xrightarrow{\simeq} \Hom_k( \scIBF_R(B), V), \quad     \be \mapsto \bar \be
\end{equation}
assigns to $\be \in \IBF_{(R,k)}(B, V)$ the unique $k$-linear map \begin{equation}\label{def:betab}
\bar \be \co \scIBF_R(B) \to V, \quad
 \bar\be(\overline{a\ot b})=\be(a,b).\end{equation}

In other words, $\scIBF_R(B)$ represents the obvious functor $\IBF(B; -) \co \kmod \to \Rmod$. The isomorphism \eqref{iso-hom-ibf-1} determines $\scIBF_R(B)$ up to a unique $k$-linear isomorphism.
We will describe $\scIBF_R(B)$ for several cases of interest in this paper. The most important  situation is captured by the following result that we state in the form of a principle.
\end{definition}

\begin{princ}\label{princip} Assume $B$ is an $R$-algebra for some $R\in \kalg$. Let $\be \in \IBF_R(B)$ be such that the induced map $\bar \be \co \scIBF_R(B) \to R$, $\overline{b_1 \ot b_2} \mapsto \be(b_1, b_2)$ is an $R$-module isomorphism.
Then for any $k$-module $V$ the map
\begin{equation} \label{princip1}
 \Hom_k(R, V) \to \IBF_{(R,k)}(B;V), \quad \vphi \mapsto \vphi\circ \be\end{equation}
is an isomorphism of $R$-modules. In particular, \begin{itemize} \item[\rm (a)] the map $R^* = \Hom_k(R, k) \to \IBF_{(R,k)}(B)$, $\vphi \mapsto \vphi \circ \be$ is an isomorphism, and

\item[\rm (b)] every $\ga\in \IBF_R(B)$ is of the form
     $\ga = r\be$ for a unique $r\in R$. \end{itemize}
\end{princ}

\begin{proof} This follows from the isomorphism \eqref{iso-hom-ibf-1}  and the equality $\be = \bar \be \circ \be_\scu$. \end{proof}

We will say that $(B,\be)$ satisfies the \emph{IBF-principle} if the assumptions and hence also the conclusions of \ref{princip} hold. Note that in this case we have a precise and explicit description of {\it all} invariant $(R,k)$-bilinear functions on $B$.\ms

\centerline{***}

\medskip

While the connection between invariant bilinear forms and centroids does not feature prominently in this paper, it has nevertheless been an important guiding principle for our work. Besides the conceptual importance of this connection, another reason for elaborating on it is Corollary~\ref{inv-cent} relating invariant forms and the centroid of an algebra. Not only will this provide the reader with a means to determine the module $\IBF_{(R,k)}(B)$, but it will also be useful in \S\ref{applisec2} when we will be looking at algebras with a ``$1$-dimensional" $\IBF_{(R,k)}(B)$.

In preparation for these results, we first present the necessary background. Using the terminology of \cite{NP}, we recall that a $(B,R)$-{\it dimodule} is an $R$-module $M$ together with $R$-bilinear maps $B \times M \to M$ and $M \times B \to M$. For example, $B$ itself is a $(B,R)$-dimodule with respect to the left and right multiplications of the algebra $B$, called the \textit{regular dimodule} and denoted $B_{\rm reg}$. Also the $R$-module $\Hom_k(B,V)$ is a $(B,R)$-dimodule with respect to the $B$-actions $\big(b_1 \cdot \vphi\big)(b_2) =  \vphi(b_2 b_1) = \big(\vphi\cdot b_2 \big)(b_1)$. For any $(B,R)$-dimodule $M$ the {\it centroid of $B$ with values in} $M$ is \[
   \cent_R(B,M) = \{ \chi \in \Hom_R(B,M) : \chi(b_1 b_2) = b_1 \cdot \chi(b_2) = \chi(b_1) \cdot b_2 \hbox{ for all } b_1, b_2\in B\}.
\]
Taking as $M$ the regular dimodule $B_{\rm reg}$, we recover the usual notion of the centroid of $B$: $\cent_R(B) = \cent_R(B, B_{\rm reg})$. We note that  $\cent_R(B)$ is isomorphic to the usual centre if $B$ is a unital algebra, where, we recall, the centre of an arbitrary algebra $B$ consists of those $c\in B$ which commute with all $b\in B$, i.e. $cb=bc$, and which associate with all $b_1, b_2\in B$, i.e., $(c,b_1, b_2)=(b_1, c, b_2) = (b_1, b_2, c) = 0$ where $(x,y,z)= (xy)z -x(yz)$.
There exists a natural map $R \to \cent_R(B)$ given by $r \mapsto \chi_r$ where $\chi_r(b) = rb$ for all $b \in B$.
We call $B$ a \textit{central} $R$-algebra if this map is an isomorphism. In this situation we will often identify $R$ with $\cent_R(B)$ without any further explicit reference.

We can now describe the connection between invariant functions and centroids.
\begin{lem} \label{inv-cent}
The restriction of the $R$-isomorphism $\sfz$ of \eqref{Rk-bil1}
to $\Cent_R\big(B,\Hom_k(B,V)\big)$ induces an $R$-linear isomorphism $\Cent_R(B,\Hom_k(B,V)) \simeq \IBF_{(R,k)}(B;V)$. In particular,
\begin{equation*} \label{inv-cent1}
   \cent_R(B,B^*) \simeq \IBF_{(R,k)}(B).
\end{equation*} \end{lem}
\begin{proof} The result is a straightforward consequence of the various definitions.
\end{proof}

\begin{cor} \label{inv-cent-co}
Assume $\be\in \IBF_{(R,k)}(B)$ is nonsingular. Then
\[ \cent_R(B) \simeq \IBF_{(R,k)}(B) \qquad (\hbox{isomorphism of
$R$-modules}).\]
If furthermore $B$ is a central $R$-algebra, then $\IBF_{(R,k)}(B)$ is a free  $R$-module of rank 1 admitting $\beta$ as a basis:
\begin{equation}\label{inv-cent-co1}
  \IBF_{(R,k)}(B) = R \beta \simeq R.
\end{equation}
\end{cor}

\begin{proof} By assumption  $\chi_\beta = \sfz^{-1}(\be)\in \Cent_R(B, B^*)$ is an isomorphism of $R$-modules. The fact that $\chi_\beta$ is centroidal amounts to saying that $B_{\rm reg} \simeq B^*$ as dimodules, whence $\IBF_{(R,k)}(B) \simeq \cent_R(B,B^*) \simeq\cent_R(B)\simeq R$. Via these isomorphism we have $\be \mapsto \chi_\be\mapsto \chi_\be^{-1}\circ \chi_\be=\Id_B \mapsto 1_R$. Since $1_R$ is a basis of $R$, $\be$ is a basis of $\IBF_{(R,k)}(B)$.\end{proof}

As we will see in \S\ref{applisec2}, there are many natural types of algebras satisfying the assumptions of Corollary~\ref{inv-cent-co}. The following result will allow us to transition from the general setting to the specific examples.

\begin{prop}\label{trans} Assume that $B$ is a central  $R$-algebra, that $\be \in \IBF_R(B)$ is nonsingular with $\be(B,B) = R$, and that $\scIBF_R(B)$ is projective. Then the IBF-principle~{\rm \ref{princip}} holds for\/ $(B,\beta)$.
\end{prop}

\begin{proof} By assumption $\bar \be$ is surjective. Denoting by $K$ the kernel of $\bar \be$, we obtain a split exact sequence
\[ 0 \to K \to \scIBF_R(B) \xrightarrow{\bar \be} R \to 0 \]
of $R$-modules and consequently a split exact sequence
\[
   0 \to \Hom_R(R,R) \xrightarrow{\hat \be} \Hom_R(\scIBF_R(B), R) \to \Hom_R(K,R) \to 0 \]
with $\hat \be (r\Id_R) = r\Id_R \circ \bar \be = r\bar \be$ for $r\in R$. The $R$-module isomorphism \eqref{iso-hom-ibf-1} sends $\bar \be$ to $\be$. We also know that $\IBF_R(B) \simeq R$ by Corollary~\ref{inv-cent-co} (applied to the case $k = R$). Hence the diagram
\[
    \xymatrix{ \Hom_R(R, R) \ar[r]^>>>>>{\hat \be}
        & \Hom_R(\scIBF_R(B), R) \ar[d]^\simeq \\
      R \ar[u]_\simeq \ar[r]^\simeq & \IBF_R(B)  }
      \qquad \qquad
      \xymatrix{ r\Id_R \ar@{|->}[r] & r\bar \be \ar@{|->}[d]
            \\ r \ar@{|->}[u] \ar@{|->}[r]
             & r\be    }
\]
commutes, proving that $\hat \be$ is an isomorphism. This in turn forces $K^{(*)}:= \Hom_R(K,R) = \{0 \}$.\footnote{To avoid any possible confusion we use $^{(*)}$ as opposed to $^*$ to denote the $R$-dual given that, by convention, $^*$ always refers to the $k$-dual.} Since $\scIBF_R(B) \simeq K \oplus R$ is projective, so is $K$. But then the canonical map $K \to (K^{(*)})^{(*)}= 0$ is injective. Thus $K= 0$. \end{proof}

\section{Functorial nature of $\scIBF$}\label{sec:fun-IBF}

As in the previous section $k$ will denote a commutative associative unital ring and $R\in \kalg$. The main purpose of this section is to describe the functorial nature of $\scIBF$ and study its behaviour under base change.

Let $(\al, f) \co (R,M) \to (S,N)$ be a morphism in $\kMOD$ (see \ref{def:cate}), i.e., $\al \co R \to S$ is a morphism in $\kalg$ and $f\co M \to N$ is $\al$-semilinear, and suppose $V$ is a $k$-module. Then for any $\ka \in \scL^2_{(S,k)}(N;V)$ the map
\begin{equation}\label{Rk-bil3}
    f^*(\ka) \co M \times M \to V, \quad (m_1, m_2) \mapsto
         \ka\big( f(m_1), f(m_2)\big)
\end{equation}
is $(R,k)$-bilinear.
We obtain in this way a $k$-linear map
\[ f^* \co \scL^2_{(S,k)}(N;V) \to \scL^2_{(R,k)}(M;V), \quad \ka \mapsto f^*(\ka).\]
If $(\be, g) \co (S,N) \to (T,P)$ is another morphism in $\kMOD$, it is immediate that $(g \circ f)^* = f^* \circ g^*$. Observe that this in particular defines a right action of the group $\textrm{GL}_R(M)$ on $ \scL^2_{(R,k)}(M;V)$ (see \ref{rem:bifun} for a functorial version of this observation).

\begin{bas-ch}\label{def:basch}
Let $\ka \co M \times M \to R$ be an $R$-bilinear form and let $\al \co R \to S$ be a morphism in $\kalg$. There exists a unique $S$-bilinear form
$\ka_\al \co  M \ot_\al S  \times M\ot_\al S \to S$ satisfying $\ka_\al (m_1 \ot s_1, m_2 \ot s_2) = \al\big(\ka(m_1, m_2)\big) s_1s_2$. In case $\al$ is clear from the context, we will denote $\ka_\al = \ka_S$ and call $\ka_S$ the {\em base change of $\ka$ by  $S$\/}.   We then have the equation $\ka_S(m_1 \ot s_1, m_2 \ot s_2) = \ka(m_1, m_2) s_1 s_2$.

Base change can also be understood in terms of the isomorphism $\sfx \co \Hom_R(M \ot_R M , R) \xrightarrow{\simeq} \scL^2_R(M)$ of \eqref{Rk-bil1}. Let $\sfx^{-1}(\ka) = \tilde{\ka} \co  M\ot_R M \to R$ be the $R$-linear map associated to $\ka$. Then $\sfx^{-1}(\ka_S) = \widetilde{\ka_S}$ is obtained from $\tilde \ka \ot \Id_S$ with the aid of two canonical $S$-module isomorphisms, namely
 \[ \widetilde{\ka_S} \co  (M\ot_R S)\ot_S(M \ot_R S) \simeq (M \ot_R M) \ot_R S \xrightarrow{\widetilde{\ka}\ot \Id_S}  R\ot_R S \simeq S. \]
The following lemma collects some results using base change. \end{bas-ch}

\begin{lem} \label{form-tr} {\rm (a) (Transitivity of base change)} Let $\ka \in \scL^2_R(M)$, let $R \xrightarrow{\al} S \xrightarrow{\be} T$ be morphisms in $\kalg$, and let $\ze \co (M \ot_\al S) \ot_\be T \to M\ot_{\be \circ \al} T$, $m\ot s\ot t \mapsto m \ot \big(\be (s)\big) t$ be the canonical isomorphism of $T$-modules. Then
\begin{equation}\label{form-tr1} \ze^* \ka_{\be \circ \al} = (\ka_\al)_\be.
\end{equation}

{\rm (b)} Let $M$ and $N$ be $R$-modules, $\la \in \scL^2_R(N)$, $f\co M \to N$ an $R$-linear map and $\al \co R \to S$ a morphism in $\kalg$. Then
\begin{equation}\label{form-tr2}
  \big(f^*(\la)\big)_\al = (f\ot \Id_S)^*(\la_\al).
\end{equation}

{\rm (c)} Let $\ka, \ka' \in \scL^2_R(M)$. Then $\ka = \ka'$ if and only if $\ka_S = \ka'_S$ for some faithfully flat extension $S\in \Ralg$. \sm

{\rm (d)} Assume that  $M$ is finitely presented. Let $\ka \in \scL_R^2(M)$ and let  $S\in \Ralg$ be such that $S$ is a flat $R$-module.
If $\ka$ is nondegenerate (resp.~nonsingular), then $\ka_S$ is nondegenerate (resp.~nonsingular). The converse holds in both cases if $S/R$ is faithfully flat. \end{lem}

\begin{proof} The proofs of (a) and (b) are immediate from the definitions.
In (c) suppose $\ka_S = \ka'_S$ for some faithfully flat $S\in \Ralg$. Then $\big(\sfx^{-1}(\ka)\big)_S = \sfx^{-1}(\ka_S) = \sfx^{-1}(\ka'_S) = \big(\sfx^{-1}(\ka')\big)_S$ as maps $(M \ot_R S)\ot_S (M\ot_R S) \to S$ by \ref{def:basch}. But then $\sfx^{-1}(\ka) = \sfx^{-1}(\ka')$ by faithfully flat descent, whence $\ka = \ka'$.

(d) Since $M$ is finitely presented and $S/R$ is flat, the canonical map $\om \co \Hom_R(M,R) \ot_R S \to \Hom_S(M \ot_R S, S)$ is an isomorphism. Recall from \eqref{Rk-bil1} the $R$-linear map $\sfz^{-1} (\ka) \co M \to \Hom_R(M, R)$. It is immediate from the definitions that
\[ \xymatrix{ M \ot_R S \ar[rr]^{\sfz^{-1}(\ka) \ot \Id_S}
          \ar[dr]_{\sfz^{-1}(\ka_S)} && \Hom_R(M, R)\ot_R S
              \ar[dl]^\om \\ & \Hom_S(M\ot_R S, S)
}\]
is a commutative diagram. Hence $\sfz^{-1}(\ka_S)$ is injective (resp.~bijective) if and only if $\sfz^{-1}(\ka)$ is so. The claim then follows from standard properties of flat (resp.~faithfully flat) extensions.
\end{proof}

\begin{prop}[$\scIBF$ and $\ibf$ as functors] \label{ibf-fun}
{\rm (a)} Let $(\al, f) \co (R, B) \to (S,C)$ be a morphism in $\kALG$. The map
\[ f \ot f \co B\ot_R B \to C \ot_S C, \quad b_1 \ot b_2 \mapsto f(b_1) \ot f(b_2) \]
is $\al$-semilinear and maps $\ibf_R(B)$ to $\ibf_S(C)$. We denote by $\ibf_\al(f) \co \ibf_R(B) \to \ibf_S(C)$ the restricted map and by $\scIBF_\al(f) \co \scIBF_R(B) \to \scIBF_S(C)$ the induced quotient map. \sm

{\rm (b} The assignments
\[ (R,B) \to \big(R,\scIBF_R(B)\big) \quad \hbox{and} \quad
     (\al,f) \mapsto \big(\al, \scIBF_\al(f)\big)
\]
define a functor $\scIBF \co \kALG \to \kMOD$ over $\kalg$ which is stable under base change in the sense of\/ {\rm \ref{fun-comm}}. \sm

{\rm (c)} The assignments
\[ (R,B) \to \big(R,\ibf_R(B)\big) \quad \hbox{and} \quad
     (\al,f) \mapsto \big(\al, \ibf_\al(f)\big)
\]
define a functor $\ibf \co \kALG \to \kMOD$ over $\kalg$ which is stable under \emph{flat} base change.
\end{prop}

\begin{proof} (a) is straightforward. We will prove (b) and (c) at the same time. It is easy to verify that $\scIBF$ and $\ibf$ are functors over $\kalg$.

We have already noted in Example \ref{derived} that
$B \mapsto B\ot_R B$ and $(\al,f) \mapsto (\al, f\ot f)$ defines a functor over $\kalg$ which is stable under base change. Indeed, for any morphism $\al \co R \to S$ the map $\nu_{B,\al}$ of \eqref{def:fun-fib1} is the well-known isomorphism
\[ \nu_{B,\al} \co (B \ot_R B ) \ot_R  S \to (B \ot_R S) \ot_S (B\ot _R S), \quad b_1 \ot b_2 \ot s \mapsto b_1 \ot 1_R \ot b_2 \ot s.\]
In the following we will abbreviate $\nu = \nu_{B, \al}$ and
$\bar \nu = \bar \nu_{B, \al} \co \scIBF_R(B) \ot_\al S \to \scIBF_S (B \ot_\al S)$. We have the following diagram with exact rows
\begin{equation} \label{flat-bc-dia}
    \vcenter{
\xymatrix{
     0 \ar@{-->}[r]^<<<<{(c)} & \ibf_R(B) \ot_R S \ar@{-->}[d] \ar[r]^{i_B \ot \Id_S} & (B\ot_R B) \ot_R S
         \ar[r]^{q_B \ot \Id_S}  \ar[d]^\nu_\simeq
         & \scIBF_R(B)\ot_R S \ar[d]^{\overline{\nu}}  \ar[r] & 0  \\
    0 \ar[r] & \ibf_S(B \ot_R S)\ar[r]^>>>>>{i_{B\ot_R S}} & (B \ot_R S) \ot_S (B \ot_R S ) \ar[r]^<<<<<{q_{B\ot_R S}} & \scIBF_S(B \ot_R S ) \ar[r] & 0
} }
\end{equation}
where the top row is obtained by tensoring \eqref{def:rep-ex} with $S$ and the bottom row is \eqref{def:rep-ex} for $B\ot_R S$ (under assumption (c), namely when $S/R$ is flat, then $i_B \otimes \Id_S$ is injective. This is reflected by the dashed line at the top left of the diagram). One easily verifies that $\nu$ sends the $S$-submodule $(i_B \ot \Id_S)(\ibf_R(B)\ot_R S)$ of $(B \ot_R B ) \ot S$ onto $\ibf_S(B\ot_R S) \subset (B \ot_R S) \ot_S (B\ot_R S)$. It follows that by restriction $\nu$ induces an $S$-module isomorphism between $(i_B\ot \Id_S)(\ibf_R(B)\ot_R S)$ and $\ibf_S(B \ot_R S)$. Since from the definitions of $\nu$ and $\overline{\nu}$ the right hand side square of diagram (\ref{flat-bc-dia}) commutes, a simple diagram chase shows that $\overline{\nu}$ is an isomorphism. In case $S$ is a flat extension, the dashed vertical arrow is injective, hence bijective.
\end{proof}

\begin{cor} \label{thm-ffdes} Let $B$ be an $S/R$-form of $A = \fra \ot_k R$ as in the descent  setting~{\rm \ref{descent-setting}}. Let  $u\in \bfAut(\fra)(S'')$ be a cocycle determining $B$ up to $R$-algebra isomorphism (see Remark~{\rm \ref{over}}). We denote by $\nu \co \scIBF_k(\fra)\ot_k S'' \to \scIBF_{S''}(\fra\ot_k S'')$ the isomorphism \eqref{def:fun-fib1} for $F=\scIBF$.
Then $\scIBF_R(B)$ is an $S/R$-form of\/  ${\scIBF(\fra\ot_k R)}$ which is isomorphic as an $R$-module to the one given by the cocycle $ \nu^{-1} \circ \scIBF (u) \circ \nu$. \end{cor}

\begin{proof} This is a special case of Theorem~\ref{mainfunctor}.
\end{proof}

\begin{rem}\label{k=R} In the above corollary we can take $k = R$. The result then applies to an arbitrary $R$-algebra $A$.
\end{rem}

\begin{lemma}\label{barcomm} Let $B$ be an $R$-algebra, $\be \in \IBF_R(B)$, and $S\in \Ralg$. \sm

{\rm (a)} We denote by $\nu \co \scIBF_R(B) \ot_R S \to \scIBF_S(B\ot_R S)$  the isomorphism \eqref{def:fun-fib1} for $F=\scIBF$, by $(\bar \be)_S$ the base change of $\bar \be \co \scIBF_R(B) \to R$ and by $\overline{ (\be_S)}$ the map \eqref{def:betab} associated to the bilinear form $\be_S$. Then, after identifying $R\ot_R S = S$, we have $(\bar \be)_S = \overline{(\be_S)} \circ \nu : $
\[ \xymatrix{\scIBF_R(B) \ot_R S \ar[r]^>>>>>{(\bar \be)_S} \ar[d]_\nu^\simeq   & R \ot_R S \ar@{=}[d]\\ \scIBF_S(B\ot_R S) \ar[r]^>>>>>>>{\overline{(\be_S)}} & S }\]

{\rm (b)} Assume that $B$ is finitely presented as a $R$-module and that $S/R$ is flat. If $(B,\be)$ satisfies the IBF-principle~\ref{princip}, then so does $(B_S, \be_S)$. The converse is true whenever $S/R$ is faithfully flat.
\end{lemma}

\begin{proof} Part (a)  is immediate from the definitions.
In (b) we know from (a) that $\overline{\beta_S}$ is an isomorphism. The assertion about faithful flatness is standard.
\end{proof}

\section{Descent of  bilinear forms}\label{sec:inv-bas}

In this section we will study the descent of bilinear forms in the setting  of \eqref{descent-setting}.  We are interested in having a full understanding of all $k$-bilinear forms on $B$. The guiding principle is that this ought to be possible if one knows the nature of the  $k$-bilinear forms of $\fra$, for example if $\fra$ is a finite-dimensional central-simple Lie algebra over a field of characteristic $0$ (Theorem~\ref{teo:Lie-desc}).

That this may be possible at all, is somehow surprising. The twisted nature and descent theory related to $B$ views $B$ {\em as an object over $R$} (note that it is not the case that $B$ is in any meaningful way a twisted form of $\fra \ot_k R$ or $\fra$ {\em as  $k$-algebras}). Yet the information that we will get is about $k$-bilinear forms $B \times B \to k$. As mentioned before, it is the $k$- (and not $R$-) bilinear forms on $B$ which are often of interest (such as in the case of infinite dimensional Lie theory as exemplified, for example,  by the affine Kac-Moody Lie algebras. See also \S\ref{applisec2} and \S\ref{sec:graded} below).

Assume that $\ka$ is a $k$-bilinear form on $\fra$. We want to know when, in a natural fashion, we can attach to $\ka$ an $R$-bilinear form $\ka_B$ on $B$.  It will be $\ka_B$ that will lead us to fully understand all $k$-bilinear functions on $B$. The key assumption that makes this construction possible and natural is that $\ka$ be invariant under algebra automorphisms. We define this concept before proceeding with the main results.

\begin{invariance}\label{def:autoinv}
For an $R$-algebra $B$ we denote by $\bfAut(B)$ the automorphism group functor of $B$. We remind the reader that $\bfAut(B)$ is the functor from the category $\Ralg$ to the category of groups that attaches to $S \in \Ralg$ the group $\Aut_S(B \ot_R S)$ of automorphisms of the $S$-algebra $B \ot_R S$, and to an arrow $ S \to T$ in $\Ralg$ and $f\in \Aut_S(B \ot_R S)$ the automorphism $f\ot \Id_T$ of $B \ot_R  T \simeq (B \ot_R S)\ot_S T$. We say that $\be \in \scL^2_R(B)$ is {\em  $\bfAut(B)$-invariant\/} if $f^*(\be_S) = \be_S$ holds for all $S\in \Ralg$ and all $f\in \bfAut(B)(S)$, where we remind the reader that $f^*(\be_S)(b_1 \ot s_1, b_2 \ot s_2) = \be_S\big(f(b_1 \ot s_1), f(b_2 \ot s_2)\big)$. In other words, $\be$ is $\bfAut(B)$-invariant if and only if $\be_S$ is $\Aut_S(B\ot_R S)$-invariant in the obvious sense for all $S\in \Ralg$.
\end{invariance}

\begin{rem}\label{rem:bifun} Automorphism invariance has a functorial interpretation. Namely, we have a functor $\scL^2(B)\co \Ralg \to \RMOD$ which assigns to the extension $S/R$ the $S$-module $\scL^2_S(B\ot_R S)$, and is given at the level of arrows by base change \ref{def:basch}. The automorphism group functor $\bfAut(B)$ acts on the functor $\scL^2(B)$ from the right. A bilinear form $\be \in \scL^2_R(B)$ is $\bfAut(B)$-invariant if and only if it is invariant under this action.

In particular, the above considerations apply to modules, viewed as trivial algebras. If $M$ is an $R$-module, the $R$-group functor $\bfGL(M)$ acts naturally on $\scL^2_R(M)$.
\end{rem}

Automorphism invariance behaves nicely with respect to base change and faithfully flat descent.

\begin{lemma}\label{lem:basch-inv} Let $B$ be an $R$-algebra, $\be \in \scL^2_R(B)$ and $S\in \Ralg$. If $\be$ is $\bfAut(B)$-invariant, then $\be_S$ is $\bfAut (B \ot_R S)$-invariant. The converse holds if $S/R$ is faithfully flat.
\end{lemma}

\begin{proof} (I) We begin with a general observation. Let $T$ be an extension of $S$. The canonical $T$-linear algebra isomorphism $\ze \co (B \ot_R S) \ot_S T \to  B\ot_R T$ of Lemma~\ref{form-tr}(a) induces a group isomorphism
\[ \Aut_T\big( (B \ot_R S) \ot_S T\big) \xrightarrow{\simeq} \Aut_T(B\ot_R T), \quad f \mapsto \ze \circ f \circ \ze^{-1}\] (we view $T$ as an object in $\Ralg$ in the obvious way). Since $(\ze^{-1})^*\big( (\be_S)_T\big) = \be_T$ by \eqref{form-tr1}, it follows that $(\be_S)_T$ is $\Aut_T\big( (B \ot_R S)\ot_S T\big)$-invariant if and only if $\be_T$ is $\Aut_T(B\ot_R T)$-invariant. \sm

(II) It is immediate from (I) and the definitions that if $\be$ is $\bfAut(B)$-invariant then $\be_S$ is $\bfAut(B\ot S)$-invariant.
\sm

(III) Assume now that $S/R$ is faithfully flat and that $\be_S$ is $\bfAut(B\ot_R S)$-invariant. To prove that $\be$ is $\bfAut(B)$-invariant, let $S'\in \Ralg$ and let $f\in \bfAut(B)(S')$. The extension $T=S\ot_R S'$ of $S'$ is faithfully flat. Hence, by Lemma~\ref{form-tr}(c), we have $f^*(\be_{S'}) = \be_{S'}$ as soon as $\big( f^*(\be_{S'})\big)_T = (\be_{S'})_T \in \scL^2_T\big( (B\ot_R S') \ot_{S'} T \big)$. Note that $\big( f^*(\be_{S'})\big)_T = (f\ot \Id_T)^*\big( (\be_{S'})_T\big)$ by \eqref{form-tr2}. Applying the considerations of (I) to the isomorphism $\ze' \co (B \ot_R S') \ot_{S'} T \to B \ot_R T$ shows that we need to prove that $\be_T$ is $\Aut_T(B\ot_R T)$-invariant. But by (I) again this is indeed the case. \end{proof}

\begin{theorem}[Descent of $\bfAut$-invariant forms]\label{prop:referee} Assume that we are under the descent setting of {\rm \eqref{descent-setting}:} $\fra$ is a $k$-algebra, $R\in \kalg$ is flat, and we are given an $R$-algebra $B$ which is a twisted form of $A = \fra\ot_k R$, hence split by some faithfully flat extension of $R$. Assume that $\ka \in \scL^2_k(\fra)$ is an $\bfAut(\fra)$-invariant bilinear form. \begin{itemize}

\item[\rm(a)] There exists a unique $R$-bilinear form $\ka_B \in \scL^2_R(B)$ such that $(\ka_B)_S= \theta^*(\ka_S)$ whenever $S/R$ is faithfully flat and $\theta \co  B \ot_R S \to \fra \ot_k S $ is an isomorphism of $S$-algebras. Moreover, $\ka_B$ is $\bfAut(B)$-invariant. \sm

\item[\rm (b)] If $\ka$ is invariant, then so is $\ka_B$. \sm

\item[\rm (c)] If $\fra$ is finitely presented and $\ka$ is nondegenerate (resp. nonsingular), then so is $\ka_B$. \end{itemize}
\end{theorem}

\begin{proof} (a) Throughout the proof we fix $S$  and $\theta$ as in (a). We let $\al \co R \to S$ be the structure map and $\ka_S$ the base change of $\ka$ to $S$. It will be convenient to first point out a general observation, which is immediate from the definitions. \begin{equation} \label{katrans}
\hbox{\it If $\be \co  S \to T$ is a homomorphism in $\kalg,$ then $(\Id_\fra \ot \be)^*(\ka_T )= \be \circ \ka_S$.} \end{equation}

We first show the existence of an $R$-bilinear form $\ka^\theta_B \in\scL^2_R(B)$ satisfying $(\ka^\theta_B)_S = \theta^*(\ka_S)$. The notation $\ka^\theta_B$ indicates that, a priori, the form depends on $\theta$ (which in turn involves an $S$). According to \ref{desformk} and \ref{over} the cocycle corresponding to the $S$-algebra isomorphism $\theta^{-1}$ is  $u = \theta_2\theta_1^{-1}$ and we have
\begin{equation}\label{aux}
\theta(B \ot 1) = \{ x \in \fra\ot_k S : u\big((\Id_\fra \ot \al_1)(x)\big) = (\Id_\fra \ot \al_2)(x)\}.
\end{equation}
We identify $B \subset B \ot_R S$ via $b \mapsto b \ot 1$ and claim that the restriction of $\theta^*(\ka_S)$ to $B \times B$, which a priori takes values in $S$, actually takes values in $R$. In other words, for $b,b'\in B$ and $x=\theta(b), x'=\theta(b') \in \fra \ot_k  S$ we claim $\theta^*(\ka_S)(b,b') = \ka_S(x , x') \in R$.  With $\al_i \co S \to S''$ as before we have, using \eqref{katrans}, \eqref{aux} and the automorphism-invariance of $\ka_{S''}$,
\begin{equation} \label{azumi1}
\begin{split}
\alpha_1 \big(\theta^*(\ka_S)(b,b')\big)
 &= \alpha_1\big(\ka_S(x, x')\big)
  = \ka_{S''}\big((\Id_\fra \ot \al_1)(x), (\Id_\fra \ot \al_1)(x') \big)\\
& = \ka_{S''}\Big( u \big( (\Id_\fra \ot \al_1)(x) \big),
       u \big( (\Id_\fra \ot \al_1)(x')\big) \Big)
\\& = \ka_{S''}\big( (\Id_\fra \ot \al_2)(x), (\Id_\fra \ot \al_2)(x') \big) \\
& = \alpha_2\big(\ka_S(x, x')\big) = \alpha_2 \big(\theta^*(\ka_S)(b,b')\big). \end{split} \end{equation}
This shows that $\theta^*(\ka_S)(b,b')$ belongs to the equalizer of $\al_1$ and $\al_2$ in $S$, but these are precisely the elements of $R$ (viewed as elements of $S$). Hence the restriction $\ka_B^\theta$ of $\theta^*(\ka_S)$ to $B$ is an $R$-bilinear form on $B$. Clearly, by its very definition, $(\ka_B^\theta)_S = \theta^*(\ka_S)$. \sm

We next aim to show that $\ka^\theta_B$ is independent of  the trivialization $\theta$. Thus, let $S'/R$ be another faithfully flat extension, say with structure map $\al' \co R \to S'$, and let $\theta' \co B \ot_R S' \to \fra \ot_k S'$ be an $S'$-algebra isomorphism. By what we just have shown, we have an $R$-bilinear form $\ka_B^{\theta'}$ satisfying $(\ka_B^{\theta'})_S = (\theta')^*(\ka_S)$, and we claim $\ka_B^\theta = \ka_B^{\theta'}$. The vehicle to  show this will be the algebra $T= S \ot_R  S'$ which we view in an obvious way as an $S$- and $S'$-algebra, say with structure maps $\be$ and $\be'$ respectively. For simplicity we denote by $\rho \co k \to R$ the structure map $\si_{R,k}$ of $R$. We have the following commutative diagram:
\[ \xymatrix{ && S  \ar[dr]^\be \\
     k \ar[r]^\rho & R \ar[ur]^\al \ar[dr]_{\al'} && T \\
     && S' \ar[ru]_{\be'}
}\]
We remind the reader that $T$ is faithful flat over $S$, $S'$ and $R$.
Let $\xi \co (\fra \ot_{\al \circ \rho} S) \ot_\be T \to \fra \ot_{\be \circ \al \circ \rho} T = \fra \ot_k T$ be the canonical $T$-linear algebra isomorphism, and define $\tilde \theta \co \fra \ot_{\be \circ \al \circ \rho} T \to B \ot_{\be \circ \al} T$ by the composition of maps
\[ \tilde \theta : \xymatrix{
   \fra \ot_{\be \circ \al \circ \rho} T  \ar[r]^{\xi^{-1}}
       & ( \fra \ot_{\al \circ \rho} S) \ot_\be T
       \ar[rr]^{(\theta \ot \Id)^{-1}} && (B\ot_\al S)\ot_\be T \ar[r]^\ze
       & B \ot_{\be \circ \al} T}.
\]
Then $\tl \theta^*$ maps $(\ka_B^\theta)_{\be \circ \al}\in \scL^2_T(B\ot_{\be \circ \al} T)$ onto a bilinear form in $\scL_T^2( \fra \ot_k T)$. In fact,
\begin{align*}
  \tl \theta^*\big( (\ka_B^\theta)_{\be \circ \al}\big) &= ( \ze \circ (\theta^{-1} \ot \Id) \circ \xi^{-1})^*\, ( (\ka^\theta_B)_{\be \circ \al}) =
  (\xi^{-1 \, *} \circ (\theta^{-1} \ot \Id)^* \circ \ze^*) \, (\ka_B^\theta)_{\be \circ \al})
  \\ &\stackrel{\eqref{form-tr1}}{=}
   ( \xi^{-1\, *} \circ (\theta^{-1} \ot \Id)^*) \, \big( (( \ka_B^\theta)_\al)_\be \big)
     = (\xi^{-1\, *} \circ (\theta^{-1} \ot \Id)^*)\big( (\theta^* \ka_{\al \circ \rho})_\be\big)
     \\ &\stackrel{\eqref{form-tr2}}{=}
     \big( \xi^{-1\, *} \circ (\theta^{-1} \ot \Id)^* \circ (\theta \ot \Id)^*\big) \, \big( ( \ka_{\al \circ \rho})_\be\big)
     = \xi^{-1\, *}\, \big( (\ka_{\al \circ \rho})_\be)
      \stackrel{\eqref{form-tr1}}{=} \ka_{\be \circ \al \circ \rho} = \ka_T.
  \end{align*}
We define $\tl \theta ' \co \fra \ot_k T \to B_{\be'\circ \al'} T$ by replacing $\theta \ot \Id$ by $\theta'\ot \Id$ in the definition of $\tl \theta$ above. By symmetry $\big( \tl \theta^{\prime \, *}\big)(\ka_B^{\theta '})_{\be ' \circ \al'}  = \ka_T$. Since $\tl \theta^{-1} \circ \tl \theta ' \in \bfAut(\fra)(T)$, we get \[ (\ka_B^{\theta '})_{\be' \circ \al'} = (\tl \theta ')^{-1 \, *} (\ka_T)
   = \tl \theta^{-1\, *} (\ka_T) = (\ka_B^\theta)_{\be \circ \al}\]
so that Lemma~\ref{form-tr}(c) implies $\ka_B^\theta = \ka_B^{\theta '}$. We are now justified to define $\ka_B = \ka_B^\theta$.

Finally, we use Lemma~\ref{lem:basch-inv} to establish that $\ka_B$ is $\bfAut(B)$-invariant: $\ka$ is $\bfAut(\fra)$-invariant $\implies \ka_S$ is $\bfAut(\fra_S)$-invariant $\implies \theta^*(\ka_S)= (\ka_B)_S$ is $\bfAut(B \ot_\al S)$-invariant $\implies \ka_B$ is $\bfAut(B)$-invariant. \sm

(b) $\ka_B$ is invariant since invariance of bilinear forms is stable under base ring extensions and, by Lemma~\ref{form-tr}(c), also under faithfully flat descent. \sm

(c) The argument is analogous to that of (b), using Lemma~\ref{form-tr}(d) and the fact that being  finitely presented is a property which is invariant under arbitrary base change and faithfully flat descent. \end{proof}

\begin{cor}\label{missing} Let $B$ be an $S/R$-form of $A = \fra \ot_k R$ as in the descent setting~{\rm \ref{descent-setting}}. We further assume that $\fra$ is finitely presented as a $k$-module and that $\ka \in \IBF_k(\fra)$ is $\bfAut(\fra)$-invariant. Let $\ka_B$ be the $R$-bilinear form on $B$ associated to $\ka$ as described in Theorem~{\rm \ref{prop:referee}}. If the IBF-principle holds for $(\fra, \ka)$, then the IBF-principle holds for $(B,\ka_B)$ a well.   \end{cor}

\begin{proof} By faithfully flat descent, the associated map $\overline{\ka_B} \co \scIBF_R(B) \to R$ is an isomorphism as soon as the extended map $(\overline{\ka_B})_{_S}$ has this property. By  Lemma~\ref{barcomm}(a) applied to $\be = \ka_B$ we have $(\overline{\ka_B})_{_S} = \overline{(\ka_B)_S}\circ \nu$. Since $\nu$ is an isomorphism, we are reduced to showing that $\overline{(\ka_B)_S}$ is an isomorphism. By the theorem, $(\ka_B)_S = \theta^*(\ka_S)$. From the definitions, the reader easily checks that $\overline{\theta^*(\ka_S)} = \overline{\ka_S} \circ \scIBF(\theta)$. By functoriality $\scIBF(\theta)$ is an isomorphism while $\overline{\ka_S}$ is an isomorphism by Lemma~\ref{barcomm}(b). Hence $\overline{(\ka_B)_S}= \overline{\ka_S} \circ \scIBF(\theta)$ is an isomorphism. This completes the proof.
\end{proof}

\section{Applications to Lie  and other classes of algebras}\label{applisec2}

We now look in detail at our general results in some important special cases. Unless stated otherwise, we use our basic setting: $k$ is a commutative associative unital ring and $R\in \kalg$. If $f$ is an endomorphism  of a finitely generated and projective $R$-module, we denote by $\tr(f)$ its trace. For details on this notion, see for example \cite[II, \S4.3]{bou:Aa} or \cite{KO}.

\subsection{Lie Algebras}\label{subsec:Lie} We start by discussing the Killing form of a Lie algebra $L$, defined by $\ka(l_1, l_2) = \tr \big( (\ad l_1) \circ (\ad l_2)\big)$ for $l_i\in L$.

\begin{prop}\label{kill-Lie} Let $L$ be a Lie algebra over $R$ whose underlying $R$-module is finitely generated and projective. \sm

{\rm (a)} The Killing form $\ka$ of $L$ is an invariant and $\bfAut(L)$-invariant $R$-bilinear form. For any $S\in \Ralg$ the base change $\ka_S$ is the Killing form of the Lie algebra $L \ot_R S$.\sm

{\rm (b)} Suppose $f$ is an $\al$-semilinear automorphism of $L$ for some $\al \in \Aut_k(R)$. Then $\ka\big( f(l_1), f(l_2)\big) = \al \big( \ka(l_1, l_2)\big)$ holds for $l_i \in L$.
\end{prop}

\begin{proof} (a) The invariance of $\ka$ follows from $\tr(f g) = \tr(g f)$ for endomorphisms $f,g$ of $L$. This identity also implies that $\ka$ is $\Aut_R(L)$-invariant: For $f\in \Aut_R(L)$ we have $\ad\big(f(l)\big) = f (\ad l) f^{-1}$ and hence $\ka\big( f(l_1) , f(l_2)\big) =  \tr\big( f (\ad l_1) (\ad l_2)  f^{-1}\big) = \ka(l_1, l_2)$. The adjoint maps of the Lie $S$-algebra $L \ot_R S$ are obtained by base change from the adjoint maps of $L$. Since the trace is invariant under base change, the Killing form of the Lie algebra $L\ot_R S$ is the base change of $\ka$ by $S$. It then follows that $\ka$ is $\bfAut(L)$-invariant. \sm

(b) We first prove an auxiliary formula. Namely, assume that $g$ is an $\al^{-1}$-semilinear endomorphism of $L$. Then \begin{equation} \label{kill-Lie1}
    \tr(f  g) = \al \big( \tr(g f) \big)
\end{equation}
(observe that both $f g$ and $g f$ are $R$-linear). By descent it is sufficient to show \eqref{kill-Lie1} in case $L$ is free of finite rank. Let $F$ and $G$ be matrices representing $f$ and $g$ in some $R$-basis of $L$. We denote by $\al\cdot G$ the matrix obtained from $G$ by applying $\al$ to all its entries. Then $f g$ and $g f$ are represented by $F( \al \cdot G)$ and $G( \al^{-1} \cdot F)$ respectively,  whence $\tr(fg) = \tr\big( F (\al\cdot G)   \big ) = \tr \big( (\al \cdot G) F \big) = \tr\big( \al \cdot (G ( \al^{-1} \cdot F)  ) \big) = \al \big( \tr(gf)\big)$. We can now establish (b): $\ka\big( f(l_1), f(l_2)\big) = \tr\big( f ( \ad l_1 \ad l_2 \, f^{-1})\big)$. Applying \eqref{kill-Lie1} with $g= \ad l_1 \ad l_2 \, f^{-1}$ shows that  $\tr\big( f ( \ad l_1 \ad l_2 \, f^{-1})\big)
= \al \big( \tr( \ad l_1 \ad l_2 \, f^{-1} f) \big) = \al\big( \ka(l_1, l_2)\big)$. \end{proof}

\begin{cor}  \label{cor:bs} In the descent setting {\rm \eqref{descent-setting}} suppose that $\fra$ is a Lie algebra whose underlying $k$-module is finitely generated and projective. Let $\ka$ be the Killing form of $\fra$. Then \sm

{\rm (a)}  $B$ is a finitely generated projective $R$-module and the unique $R$-bilinear form $\ka_B$ on $B$ associated to $\ka$ in Theorem~{\rm \ref{prop:referee}} is the Killing form of the Lie algebra $B$. If  $B$ is realized as an $R$-subalgebra of $ \fra \ot_k S$ as explained in \eqref{desformk},\footnote{Which is always possible up to $R$-isomorphism -- this is the content of (\ref{desformk})} the form $\ka_B$ is the restriction of the Killing form $\ka_S$ of $\fra \ot_k S$ to $B$.
\sm

{\rm (b)} If $\ka$ is nonsingular, then the Killing form of $B$ is non-singular. \sm

{\rm (c)} If $\ka$ is nonsingular and $\fra$ is a central $k$-algebra, then
$B$ is a central $R$-algebra, and $\IBF_R(B)$ is free $R$-module of rank 1 admitting $\ka_B$ as a basis. In particular $\IBF_R(B) = R  \ka_B$. \end{cor}

\begin{proof} (a) By Proposition~\ref{kill-Lie}, $\ka_S$ is the Killing form of the $S$-Lie algebra $\fra \ot_k S$. The property of being finitely generated and
projective is stable under arbitrary base change and faithfully flat descent. Hence the $R$-module $B$ is finitely generated and projective. By Proposition~\ref{kill-Lie}, $\be_S$ is the Killing form of $B_S$. Let $\theta \co B \ot_R S \to \fra \ot_k S$  be a trivialization. Since the isomorphism $\theta $ preserves Killing forms, we get $\be_S = \theta^*(\ka_S)$. Now (a) follows from the uniqueness assertion in Theorem~\ref{prop:referee}.

(b) This follows from (a) and Lemma~\ref{form-tr}(d) (we remind the reader that every finitely generated projective module is finitely presented).

(c) By  \cite[Lemma~3.1]{P} the $S$-algebra $\fra \ot_k S$ is central. The faithfully flat descent reasoning of \cite[Lemma 4.6(3)]{GP} then shows that the $R$-algebra $B$ is a central. The last claim now follows from Corollary~\ref{inv-cent-co}. \end{proof}

At this point a very natural question arises: What are interesting examples of Lie algebras for which the IBF-principle \ref{princip} holds with respect to the Killing form? To convince the reader that (over rings) one cannot expect easy answers, we will look at one of the simplest and innocent looking Lie algebras.

\begin{example}[$\lsl_2(k)$]\label{ex:sl2} Let $\lsl_2(k)$ be the Lie algebra of all traceless $2\times 2$-matrices with entries in our ring $k$. Its underlying module is free of rank $3$, with the following matrices forming the standard basis:
\[ e = \begin{pmatrix}0 & 1 \\ 0 & 0 \end{pmatrix},  \quad
h = \begin{pmatrix}1 & 0 \\ 0 & -1 \end{pmatrix},  \quad
f = \begin{pmatrix}0 & 0 \\ 1 & 0 \end{pmatrix}. \]
A straightforward calculation shows that $\ibf_k\big(\lsl_2(k)\big)$ is spanned by
\[ h \ot e, \;  e\ot h, \; f\ot h, \; h \ot f, \; 2 e\ot e, \; 2 f\ot f, \; h \ot h - 2 f\ot e, \; h \ot h - 2 e \ot f. \]
Hence
\[ \scIBF_k\big( \lsl_2(k)\big) = \big((k/2k) \overline{e\ot e} \big)
  \; \oplus \; \big((k/2k) \overline{f \ot f} \big)\; \oplus \; \Span_k\{ \overline{e\ot f}, \overline{f \ot e}\}.
\]
Consequently: \begin{itemize}
  \item If $2k=0$, then $\scIBF_k\big( \lsl_2(k)\big)$ is free of rank $4$, with basis $\{\overline{e\ot e},\, \overline{f\ot f},\, \overline{e\ot f}, \, \overline{f\ot e}\}$.
  \item If $2\in k^\times$, i.e., $2\in k$ is invertible, then $\scIBF_k\big( \lsl_2(k)\big)$ is free of rank $1$, e.g.\ with basis $\{\overline{h\ot h}\}$.
\end{itemize}
Using the isomorphism \eqref{iso-hom-ibf-1} and the description of $\scIBF\big( \lsl_2(k)\big)$ above, we can define an invariant bilinear form $\ga\in \IBF_k \big( \lsl_2(k)\big)$, sometimes called the \textit{normalized Killing form} or the \textit{normalized invariant form}, by
\[ \ga(e, e ) = 0 = \ga( f , f) , \quad
      \ga( e, f) = 1 = \ga( f, e). \]
Note that $\ga(h, h) = 2$ and that all other values of $\ga$ on the standard basis of $\lsl_2(k)$ are zero, in particular $\ga$ is symmetric.
The description of $\scIBF_k\big( \lsl_2(k)\big)$ above implies that:
\[\hbox{\it If $2$ is not a zero divisor in $k$, then
    $\IBF_k\big( \lsl_2(k)\big) = k  \ga$ is free of rank $1$.}\]
Moreover, by calculating the discriminant of $\ga$ one obtains:
\[ \hbox{\it $\ga$ is nonsingular} \iff 2\in k^\times \iff
 \hbox{\it the IBF-principle holds for $(\lsl_2(k), \ga)$.}\]
In this case, $\ga$ is $\bfAut\big(\lsl_2(k)\big)$-invariant, which can be seen by noting that $(\ad x )^3 - 2\ga(x,x) \ad x= 0$ is the generic minimal polynomial of $\lsl_2(k)$. \sm

It is straightforward that  $12\ga$ is the Killing form of $\lsl_2(k)$. In particular, the Killing form vanishes if $2k=0$ (not surprising since then $\lsl_2(k)$ is a $2$-step nilpotent Lie algebra) or if $3k=0$ (somewhat surprising since $\lsl_2(k)$ is a simple Lie algebra when $k$ is a field of characteristic $3$). The conclusion is that for the setting of this paper the normalized Killing form $\ga$ is better behaved than the Killing form itself. A case in point is a revised version of Corollary~\ref{cor:bs} for $\fra=\lsl_2(k)$ and $2\in k^\times$. Since $\lsl_2(k)$ is then central,\footnote{Centrality only requires that $2$ not be a zero divisor.} the proof of loc.\ cit.\ shows that:
\begin{itemize}
  \item If $2\in k^\times$, any $S/R$-form $B$ of $\lsl_2(R)$ is central and  has a nonsingular invariant bilinear form $\be$ (not necessarily the Killing form), for which $ \IBF_R(B) = R \be$ is free of rank $1$.
\end{itemize}
\sm

\noindent It is instructive to summarize what we have shown above for the  special case $k=\ZZ$. \begin{itemize}
  \item The $k$-module $\scIBF_\ZZ\big( \lsl_2(\ZZ)\big)$ is neither projective nor cyclic, in particular, the IBF-principle does not hold for $\lsl_2(\ZZ)$. Yet $\IBF_\ZZ\big( \lsl_2(\ZZ) \big) = \ZZ  \ga$ is free of rank $1$.

   \item All invariant bilinear forms of $\lsl_2(\ZZ)$ are symmetric (even though $\lsl_2(\ZZ)$ is not perfect, cf.\ Remark~\ref{rem:perfect}).
   \item All non-zero invariant bilinear forms of $\lsl_2(\ZZ)$ are nondegenerate, but none of them is nonsingular.
\end{itemize}
\end{example}

\begin{rem}[Generalizations of Example~\ref{ex:sl2}]
We note that Example~\ref{ex:sl2} can be generalized by replacing $\lsl_2(k) = \lsl_2(\ZZ) \ot_\ZZ k$ by $\scG \ot_\ZZ k$ where $\scG$ is the Lie algebra of a split simple simply-connected Chevalley-Demazure group scheme. In terms of Lie algebras, $\scG$ is a Chevalley order of a split simple Lie algebra $(\g, \frh)$ over $\QQ$, say with root system $\De$, which is compatible with the root space decomposition of $(\g, \frh)$ and satisfies $\scG \cap \frh= \Span_\ZZ\{ h_\al : \al \in \De\}$ (with the standard notation). In this setting the existence of an invariant bilinear form $\ga$ as above follows from \cite{SS, GN}. It is uniquely determined by the condition $\ga(h_\al, h_\al) = 2$ for any long root $\al$.  Details will be left to the reader. \end{rem}

In what follows we restrict our presentation to base fields of characteristic $0$. To abide by standard notation we denote our algebra $\fra$, which is now a finite-dimensional semisimple Lie algebra defined over a field $k$ of characteristic $0$, by $\g$. We are interested in twisted forms of $\g \ot_k R$ for some $R\in \kalg$. In the case when $k$ is algebraically closed, $\g$ is simple and $R$ is the Laurent polynomial ring $k[t_1^{\pm 1}, \dots ,t_n^{\pm1}]$, the twisted forms in question are related to the affine Kac-Moody Lie algebras (the case $n=1$) and more generally to multiloop algebras (see \cite{ABFP2,GP,n:persp, P} and \S\ref{sec:graded} for further details and references).

\begin{theorem}\label{teo:Lie-desc}
Let $\g$ be a finite-dimensional semisimple Lie algebra over a field $k$ of characteristic $0$.  Let $B $ be a twisted form of $\g\ot_k R$, split by a faithfully flat extension $S/R$. \sm

{\rm (a)} Then $B$ is a finitely generated projective $R$-module and perfect as a Lie algebra. The Killing form of the $R$-algebra $B$ coincides with the bilinear form $\ka_B$ associated to the Killing form $\ka$ of $\g$ in Theorem~{\rm \ref{prop:referee}}. In particular the Killing form of $B$  is nonsingular and $\bfAut(B)$-invariant. If $B$ is realized as an $R$-subalgebra of $ \g \ot_k S$, the Killing form of $B$ is the restriction of the Killing form $\ka_{\g \ot S}$ to $B$. \sm

{\rm (b)} Assume henceforth that $\g$ is central, hence central-simple. Then $B$ is a central $R$-algebra, $\IBF_R(B)$ is a free $R$-module of rank $1$ admitting $\ka_B$ as a basis, and $(B, \ka_B)$ satisfies the IBF-principle~{\rm \ref{princip}}.
\end{theorem}

\begin{proof} With the exception of the perfectness statement, (a) and the first part of (b) is a re-statement of Corollary~\ref{cor:bs}. But being perfect is a property (of arbitrary algebras) which is stable under arbitrary base change and faithfully flat descent. Since $B \ot_R S \simeq \g \ot_k S$ and the latter is perfect, $B$ is perfect. That $(B,\ka_B)$ satisfies the IBF-principle is a special case of Corollary~\ref{missing} keeping in mind that $(\g, \ka)$ satisfies the IBF-principle in view of Proposition~\ref{trans}.
\end{proof}

\begin{rem}\label{zelma} In the special case when $R=S$ and $B=\fg\ot_k R$, our result says that
\begin{equation} \label{zelma1}
    \scIBF(\g \ot_k R) \simeq R \simeq \scIBF_k(\g) \ot_k R,
    \end{equation}
This formula is also a special case of \cite[Th.~4.1]{zus} which, using methods from Lie algebra homology,  determines the ``predual" of the space of symmetric bilinear forms of a Lie algebra of type $L\ot_k R$ for $k$ a field of characteristic $\ne 2$ and $L$ any Lie algebra over $k$. Assuming for comparison reasons that $L$ is perfect, we know (Remark~\ref{rem:perfect}) that all invariant bilinear maps are symmetric and thus \cite[Th.~4.1]{zus} becomes
\eqref{zelma1} for Lie algebras of type $L\ot_k R$. We emphasize that the approach of \cite{zus} {\it cannot} be applied to the case of twisted forms of $\g \ot_k R$. Developing methods that would apply to these algebras was the original motivation for our work. As already observed, such twisted algebras already arise in the affine Kac-Moody setting and are crucial for EALA theory.

An immediate consequence of \eqref{zelma1} is that every invariant bilinear form $\be\in \IBF_{(R,k)}(\g \ot_k R) = \IBF_k(\g \ot_k R)$ has the form $\vphi \circ \ka_R$ for a unique $\vphi \in R^*=\Hom_k(R,k)$, i.e., $\be(x_1 \ot r_1, x_2 \ot r_2) = \ka (x_1, x_2)\, \vphi(r_1 r_2)$ for $x_i \in \g$ and $r_i \in R$. This latter fact has recently been re-proven in \cite[Lemma~2.3]{MSZ} in case $k$ is an algebraically closed field of characteristic $0$, using the structure theory of $\g$.

The untwisted case was, out of necessity, the first objective of our work. The methods to be developed, however, had to be compatible with descent theory so that results about twisted algebras could be obtained. In retrospect, we ``knew"  that the functor on $k$-spaces $\IBF_{k}(\fg \otimes_k S,-)$ is represented by $S$ (one can reinterpret \cite{zus} or \cite{MSZ} this way). But how does one recover $S$ from $\fg \otimes_k S$? The answer is as its centroid. By descent, the centroid of $B$ is in this case naturally isomorphic to $R$. Our result shows that the representability of $\IBF_k(B,-)$ in terms of the centroid is indeed the correct point of view.
\end{rem}

\begin{rem}  For crucial use in \cite{pps-new}, we note the following.
Since the IBF-principle holds by Theorem~\ref{teo:Lie-desc}(b), composing the isomorphism \eqref{princip1} with the inverse of the isomorphism given in Lemma \ref{inv-cent}, we have an isomorphism $\Hom_k(R,V)\to  \Cent_R\big(B,\Hom_k(B,V)\big)$,  $\vphi\mapsto \widetilde\vphi$ such that $ \widetilde\vphi(b)(b')=\vphi\big(\ka_B(b,b')\big)$.
\end{rem}

\subsection{Unital algebras} In this subsection we will discuss invariant bilinear forms of unital algebras $B$ defined over some $R\in \kalg$. To do so, we will use the {\em associator module\/} and {\em commutator module\/} defined for an arbitrary algebra $B$ by
\[ (B,B,B) = \Span_\ZZ\{ (a,b,c,) : a,b,c\in B \} \quad \hbox{and} \quad
[B, B] = \Span_\ZZ\{ [a,b] : a,b\in B\}
\]
respectively, where $(a,b,c) = (ab)c - a (bc)$ is the associator and $[a,b] = ab-ba$ is the commutator in $B$.\footnote{If $B$ happens to be a Lie algebra, the commutator as defined here is twice the Lie algebra product. This notational conflict should not cause any problems since in the following we will employ the notation $[a,b]$ for non-Lie algebras only.} It is immediate that $(B,B,B)$ and $[B,B]$ are $R$-submodules of $B$. We define
\[ \ac(B) = (B,B,B) + [B,B] \quad \hbox{and} \quad \AC(B) = B / \ac(B).\]
Let $1_B \in B$ be the identity element of $B$. Thus $b \, 1_B  = b = 1_B \, b$ for all $b\in B$. A unital algebra is perfect, whence $\IBF_{(R,k)} (B; V) = \IBF_k(B; V)$ for any $k$-module $V$ by Remark~\ref{rem:perfect}.\sm

For convenience for the remainder of this section we will denote the identity element of $B$ by $1$.

\begin{lem}\label{lem:uni-gen} Let $B$ be a unital $R$-algebra. Then the multiplication map $\mu \co B \ot_R B \to B$, $\mu(a \ot b) = ab$, induces an isomorphism
\[ \bar \mu \co \scIBF_R(B) \to \AC(B), \quad \bar \mu(\overline{a\ot b}) = \overline{ ab} \]
with inverse given by $\bar a \mapsto \overline{1 \ot a} = \overline{a \ot 1}$. Hence, for any $k$-module $V$ the natural map
\[ \Hom_k(\AC(B), V) \to \IBF_k(B;V), \]
which assigns to $\vphi \in \Hom_k(\AC(B), V)$ the bilinear function $(a,b) \mapsto \vphi(\overline{ab})$, is an isomorphism of $R$-modules. Its inverse is given by assigning to $\be$ the linear function $\bar b \mapsto \be(b,1)$, where $b\in B$.
\end{lem}

\begin{proof} By \eqref{def:repr-0}, $\ibf_R(B)$ is spanned by elements of the form $ab \ot c - a \ot bc$ and $a\ot b - b \ot a$. It is clear that $\mu\big(\ibf_R(B)\big) = \ac(B)$. Hence $\bar \mu$ is well-defined and surjective.
Let $\nu\co B \to B \ot_k B$ be defined by $\nu(a) = 1  \ot a$. Then we have
\begin{align*}
  \nu\big( (a,b,c)\big) &= 1  \ot (ab)c - 1  \ot a(bc) \equiv
        ab \ot c - a \ot bc \equiv 0 \mod \ibf_R(B), \quad \hbox{and}\\
   \nu\big( [a,b]\big) &= 1  \ot ab - 1  \ot ba \equiv a\ot b - b \ot a
 \equiv 0 \mod \ibf_R(B).
\end{align*}
We thus get a well-defined $k$-linear map $\bar \nu \co \AC(B) \to \scIBF_R(B)$  satisfying $\bar \nu( \bar b) = \overline{b\ot 1} = \overline{1\ot b}$. Because of $(\nu \circ \mu)(a \ot b) = 1  \ot ab = (1  \ot ab - 1  a \ot b) + a \ot b \equiv a \ot b \mod \ibf_R(B)$, we have $\bar \nu \circ \bar \mu = \Id_{\scIBF(B)}$, proving injectivity and thus bijectivity of $\bar \mu$.
Under the isomorphism $\bar \mu$, the universal bilinear map $\be_\scu\co B \times B \to \scIBF_R(B)$ becomes $\be_{\scu, u} \co B \times B \to \AC(B)$ with $\be_{\scu,u}(a,b) = \overline{ab}$. In view of \eqref{iso-hom-ibf-2} this implies the last claim.\end{proof}

\begin{cor}\label{cor:unit}
Let $B$ be a unital $R$-algebra and assume that $B = Rb_0 \oplus \ac(B)$ for some $b_0 \in B$ where $R b_0$ is free with basis $\{b_0\}$. Let $\pi \co B \to R$ be defined by $b = \pi(b) b_0 \oplus b_\ac$ where $b_\ac \in \ac(B)$, and define \[ \be_0 \co B \times B \to R, \qquad \be_0(a,b) = \pi(ab).\]
Then $\be_0 \in \IBF_R(B)$, and $(B,\be_0)$ satisfies the IBF-principle. Furthermore
\[ \tilde \be_0 \co \AC(B) \to R, \quad \bar b  \mapsto \be_0(b,1)
\]
is a well-defined $R$-module isomorphism. \end{cor}

\begin{proof} The proof is straightforward.
\end{proof}

\subsection{Azumaya algebras} \label{subsec:Azu}
This subsection fits within the descent setting \eqref{descent-setting}: $k$ is base ring, $R\in \kalg$ is flat as a $k$-module (for example $k=R$), and $S\in \Ralg$ is a faithfully flat $R$-module. We let $\fra= M_n(k)$. Then our $S/R$-form $B$ of $M_n(k) \ot_k R = M_n(R)$ is an Azumaya algebra over $R$ of constant rank $n^2$.

We start by recording some facts about $\fra$. As a unital algebra, $\fra$ is perfect. It is also well-known that $\fra$ is central. Moreover, $\fra$ has a natural invariant bilinear form, the trace form $\ka$ defined by
\[
           \ka(x,y) = \tr(xy)
\]
where this last is the usual trace of the matrix $xy$. It is easy to see (using the standard dual basis of the elementary matrices $E_{ij}$) that $\ka$ is nonsingular. Moreover, $\ka$ is $\bfAut(\fra)$-invariant. Indeed, since $\fra \ot_k K = M_n(K)$ for any $K\in \kalg$, it suffices to verify that $\ka$ is automorphism-invariant. Thus let $\si \in \Aut_k(\fra)$ and $x,y \in \fra$.  To show that $xy$ and $\si(x)\si(y)$ have the same trace, it is enough to prove that for all $\fp \in \Spec(k)$  the two elements $(xy)_\fp$ and $\big( \si(x)\si(y)\big)_\fp$ of $M_n(k_\fp)$ have the same trace. Clearly $(xy)_\fp = x_\fp y_\fp$ and $\big( \si(x)\si(y)\big)_\fp  =\si_\fp(x_\fp) \si_\fp(y_\fp)$ where $\si_\fp = \si \ot \Id_{k_\fp}$. We may therefore assume that $k$ is a local ring. But then, by the Skolem-Noether Theorem for local rings (\cite[IV, Cor.~1.3]{KO}), $\si$ is given by conjugation by an invertible matrix $M\in \GL_n(k)$, whence $\si(x)\si(y) = M xy M^{-1}$, and so clearly $xy$ and $\si(x)\si(y)$ have the same trace. We also have
\[ \fra = k E_{11} \oplus [\fra,\fra], \quad [\fra,\fra]= \{ x\in \fra : \tr(x) = 0\} = \ac(\fra)\]
since any $x=\sum_{i,j} x_{ij} E_{ij}$ can be uniquely written as
\begin{equation} \label{eq:dec}  x= \big(x_{11} + \tsum_{1<i} x_{ii} \big) E_{11} + \tsum_{1 < i} x_{ii} (E_{ii} - E_{11}) + \tsum_{i \ne j} x_{ij} E_{ij}\end{equation}
and $[\fra,\fra]$ is spanned by matrices of type $[E_{ii}, E_{ij}]=E_{ij}$ and $[E_{ij}, E_{ji}] = E_{ii} - E_{jj} = (E_{ii} - E_{11}) - (E_{jj}-E_{11})$ for $i\ne j$. Formula \eqref{eq:dec} implies that the trace form $\ka$ is the bilinear form of Corollary~\ref{cor:unit}. Hence
$(\fra, \ka)$ satisfies the IBF-principle.

\begin{theorem} \label{teo:Azu}
Let $B$ be an $S/R$-form of $M_n(R)$ and let $\ka_B$ be the bilinear form  associated to the trace form $\ka$ of $M_n(k)$ in Theorem~{\rm \ref{prop:referee}}. \sm

{\rm (a)} Then $\ka_B$ is a nonsingular, invariant and $\bfAut(B)$-invariant bilinear form and a basis of $\IBF_R(B)$. \sm

{\rm (b)} $\scIBF_R(B) \simeq \AC(B) \simeq R$, and the map $\scIBF_R(B) \mapsto R$, $\overline{b \ot b'} \mapsto \ka_B(b,b')$ is an isomorphism of $R$-modules. Hence $(B,\ka_B)$ satisfies the IBF-principle~{\rm \ref{princip}}. \sm

{\rm (c)} If $B$ is realized as an $R$-subalgebra of $M_n(S)$, see Remark~{\rm \ref{over}}, $\ka_B$ coincides with the restriction of the trace form of $M_n(S)$ to $B$.
\end{theorem}

\begin{proof}
(a) and (b) follow from Theorem~\ref{prop:referee}, Corollary~\ref{missing} and Corollary~\ref{cor:unit}. For the proof of (c) one uses the automorphism invariance of the trace of $M_n(S'')$ and the reasoning in \eqref{azumi1} to conclude that the restriction $\la$ of the trace form of $M_n(S)$ to $B$ has values in $R$. Since $\la_S$ is the trace form of $M_n(S)$ and thus coincides with $(\ka_B)_S$, we get $\la = \ka_B$ from uniqueness in Theorem~\ref{prop:referee} (or from Lemma~\ref{form-tr}(c)).
\end{proof}

\begin{rem}
The form $\ka_B$ is, by definition, nothing but the reduced trace form of the Azumaya algebra $B$ as defined in \cite{KO}. This proves (without the construction of the characteristic polynomial as done in \cite{KO}) that the reduced trace form, which a priori takes values in $S$, does take values in $R$.
\end{rem}

\begin{cor} Every Azumaya algebra $B$ over $R$ has a nonsingular, invariant and $\bfAut(B)$-invariant bilinear form $\ka_B$ such that $(B,\ka_B)$ satisfies the IBF-principle~{\rm \ref{princip}}. In particular, $\IBF_R(B)$ is a free $R$-module with basis $\{\ka_B\}$. \end{cor}

\begin{proof} If $B$ has constant rank, then $B$ is an $S/R$-form of some $M_n(R)$ and the results follows from Theorem~\ref{teo:Azu}. In general, we can decompose the identity element $1_R$ of $R$ into a sum $1_R = e_1 + \cdots + e_s$ of orthogonal idempotents $e_i\in R$ such that $B= B_1 \boxplus \cdots \boxplus B_s$ is a direct product of ideals $B_i=e_i B$, each $B_i$ is an Azumaya algebra of constant rank $\rho_i$ over $R_i = e_i R$, and $\rho_i\ne \rho_j$ for $i \ne j$ (if $\rho_i = \rho_j$ then we replace $e_i$, $e_j$ by $e_i + e_j$).
 We then define $\ka_B$ as the orthogonal sum of the forms $\ka_{B_i}$ constructed previously. Nonsingularity follows from $\Hom_R(B_1 \boxplus \cdots \boxplus B_s, R) \simeq \bigoplus_{i=1}^s \Hom_{R_i}(B_i, R_i)$
and the nonsingularity of the $\ka_{B_i}$.
Finally, $\bfAut(B)$-invariance holds since the decomposition $B= B_1 \boxplus \cdots \boxplus B_s$ is preserved under base ring extensions and automorphisms.
\end{proof}

\subsection{Octonion algebras} \label{sec:octo}
As in the previous subsection, $k$ here is an arbitrary base ring and $R\in \kalg$. Following \cite{Bix,LoPeRa,Petersson} we call an algebra $B$ over $R$ an {\em octonion algebra} if its underlying $R$-module is  projective of constant rank $8$, contains an identity element $1_B$, and admits a quadratic form $n_B \co B \to R$, the {\em norm of $B$\/}, satisfying the following two conditions. \begin{enumerate}
  \item[(i)] The associated bilinear form $n_B \co B \times B \to R$, $n_B(a,b) = n_B(a+b) - n_B(a) - n_B(b)$, is nonsingular, and

   \item[(ii)] $n_B(ab) = n_B(a)\, n_B(b)$ holds for all $a,b\in B$.
\end{enumerate}
For an octonion algebra $B$ the linear form $t_B = n_B(1_B, -)$
is called the {\em trace of $B$\/}. An example of an octonion algebra is the algebra $\Zor(R)$ of {\em Zorn vector matrices\/}, defined on the $R$-module
\[    \rmZ = \Zor(R) = \begin{bmatrix}
   R & R^3 \\ R^3 & R \end{bmatrix} \]
with product
\[
    \begin{bmatrix}
      \al_1 & u \\ x & \al_2 \end{bmatrix} \,
    \begin{bmatrix}
      \be_1 & v \\ y & \be_2 \end{bmatrix} =
    \begin{bmatrix}
      \al_1 \be_1 -   {^t u} y & \al_1 v + \be_2 u + x \times y \\
      \be_1 x + \al_2 y + u \times v & - {^t x}v + \al_2 \be_2
    \end{bmatrix}
\]
for $\al_i, \be_i \in R$ and $u,v,x,y\in R^3$. Here $^t u y$ and $x\times y$ are the usual scalar and vector product of vectors in $R^3$.
For this octonion algebra and $a = \left[\begin{smallmatrix}
  \al_1 & u \\ x & \al_2 \end{smallmatrix}\right]\in \rmZ$ one has
\[ 1_{\rmZ} = \begin{bmatrix} 1 & 0 \\ 0 & 1 \end{bmatrix}, \quad
  n_{\rmZ}(a) = \al_1 \al_2 + {^t u x}, \quad \tr_{\rmZ}(a) = \al_1 + \al_2.\]
For our approach to octonion algebras it is important that an $R$-algebra $B$ is an octonion algebra if and only if there exists a faithfully flat (even faithfully flat and \'etale)
$S\in \Ralg$ such that $B \ot_R S \simeq \Zor(S)$ (\cite[Cor.~4.11]{LoPeRa}). The algebra $\Zor(R)$ is referred to as {\em split octonions}.

\begin{theorem}\label{lem:zor} Let $B$ be an octonion algebra over $R$. \sm

{\rm (a)} $B$ is a central $R$-algebra satisfying $\ac(B) = (B,B,B) = [B,B]$. \sm

{\rm (b)} The bilinear form $\ta \co B \times B \to R$, defined by $\ta(x,y) = \tr(xy)$, is an invariant, nonsingular and $\bfAut(B)$-invariant $R$-bilinear form. It coincides with the form $\ta_B$ associated in Theorem~{\rm \ref{prop:referee}} to the bilinear form $\ta$ of $\Zor(R)$ with respect to any splitting $B \ot_R S \simeq \Zor(S)$ of $B$. \sm

{\rm (c)} $(B,\ta)$ satisfies the IBF-principle.
\end{theorem}

\begin{proof} The algebra $B$ fits into our descent setting with $k=R$ and $\fra = \Zor(R)$. We first prove all assertions for $\fra$. Straightforward calculations  (admittedly tedious in the case of the associator module) show that
\begin{enumerate}
  \item[(i)] $[\fra, \fra]=(\fra, \fra, \fra)= \ac(\fra)$, $\fra= R e \oplus \ac(\fra)$ for $e=\left[ \begin{smallmatrix} 1 & 0 \\ 0 & 0 \end{smallmatrix} \right]$, $\fra$ is central.

  \item[(ii)] For $a,b$ as in the product formula above we have $\ta(a,b) = \al_1 \be_1 + \al_2 \be_2 - {^t u} y - {^t x} v$. Hence $\ta$ is nonsingular and symmetric.
 \item[(iii)] $\ta \big(\ac(\fra)\big) = 0$, $\ta(e)=1$, whence $\ta$ is the bilinear form associated to the decomposition $\fra = Re \oplus \ac(\fra)$ in Corollary~\ref{cor:unit}.
\end{enumerate}
This corollary now implies that $(\fra, \ta)$ satisfies the IBF-principle. To establish that $\ta$ is $\bfAut(\fra)$-invariant, we recall that since $\fra$ is a quadratic algebra the trace linear form $\tr$ of $\fra$ is uniquely determined by the unital algebra $\fra$ (\cite[Lem.~1.2]{Petersson}). For any extension $K\in \kalg$, the base change $\ta_K$ is therefore the bilinear form $\ta_{\fra \ot K}$ of the $K$-algebra $\fra\ot_R K = \Zor(K)$. Uniqueness of the trace then implies that $\tr_K$ is $\Aut_K(\fra_K)$-invariant.

We now consider an arbitrary octonion algebra $B$ over $R$ and choose a faithfully flat extension $S\in \Ralg$ such that $B \ot_R S \simeq \Zor(S)$ as $S$-algebras. In the first part of the proof we have established all claims for $\fra= \Zor(R)$, whence also for $\Zor(S)$. The assertions in (a) now hold for $B$ since they are all preserved by faithfully flat descent. In (b) it suffices to establish the second part, but this follows from the fact that the base change of the trace form $\ta$ of $B$ to $S$ is the trace form of $\Zor(S)$.  (c) is a special case of
Corollary~\ref{missing}. \end{proof}

\begin{rem}[Quadratic algebras] The experts will undoubtedly have noticed that the automor\-phism-invariance of the bilinear form $\ka_B$ comes from the fact that octonions are quadratic algebras, see e.g.\ \cite[1.1]{Petersson}. Hence our techniques can also be applied to certain quadratic algebras whose trace forms are invariant.
\end{rem}

\subsection{Alternative algebras}\label{susec:alt}
We consider alternative algebras, always assumed to be unital, over some base ring $R$. Recall (\cite{Bix}) that an alternative algebra $B$ is called {\em separable\/} if for every algebraically closed field $K$ in $\Ralg$ the $K$-algebra $B\ot_R K$ is finite-dimensional and a direct sum of simple ideals. Equivalently, the unital universal multiplication envelope of $B$ is a separable associative algebra. By \cite[Prop.~2.11]{Bix}, an $R$-algebra $B$ is central separable and alternative if and only if $R=R_1 \boxplus R_2$ is a direct sum of two ideals such that $B_1= R_1B$ is an Azumaya algebra over $R_1$ and $B_2=R_2B$ is an octonion algebra over $R_2$. It is now straightforward to extend the results of \S\ref{subsec:Azu} and \S\ref{sec:octo} to central separable alternative algebras. We leave the details to the reader and only mention the following.

\begin{cor}\label{cor:cs-alt} A central separable alternative $B$ over $R$ has a nonsingular invariant bilinear form $\ka_B$ such that\/ $\IBF_R(B)$ is a free $R$-module with basis $\{\ka_B\}$.
\end{cor}

\subsection{Jordan algebras}\label{subsec:Jor}
Central separable Jordan algebras over rings $R$ containing $\frac{1}{2}$ (\cite{Bix:Ja,loos:sep}) are another class of algebras to which our results apply. 

By definition, a unital Jordan algebra $J$ over $R$ is separable if and only if $J\ot_R K$ is finite-dimensional semisimple for all fields $K\in \Ralg$. A central separable Jordan algebra $J$ is generically algebraic (\cite[Ex.~2.4(d)]{loos}). Let $\tr\in \Hom_R(J,R)$ be its generic trace. By Prop.~2.7 of loc.\ cit.\ the associated bilinear form $\ta$, defined by $\ta(a,b) = \tr(ab)$, is invariant, $\Aut(J)$-invariant and commutes with extensions and faithfully flat descent. Since separability and being generically algebraic is invariant under base ring extensions, $\ta$ is in fact $\bfAut(J)$-invariant. By \cite[Cor.~16.16]{loos:jp}, $\ta$ is nondegenerate for separable Jordan algebras over fields. From Lemma~\ref{ns-char} we then get that $\ta$ is nonsingular.

\begin{lem}\label{ns-char} Let $M$ be a finitely generated projective $R$-module and let $\be \in \scL^2_R(M)$. Then $\be$ is nonsingular if and only if $\be_K$ is nondegenerate for all $K\in \kalg$ which are fields.
\end{lem}

\begin{proof} This is an application of \cite[II, \S3.3 Th.~1, \S3.2 Cor. de la Prop.~6 and \S5.3 Th.~2]{bou:ACa}. \end{proof}

Thus, in view of Corollary~\ref{inv-cent-co} we have the following.

\begin{theorem}\label{prop:Jcs}
The generic trace form $\ta$ of a central separable Jordan algebra $J$ over a ring $R$ containing $\frac{1}{2}$ is an invariant nonsingular and  $\bfAut(J)$-invariant bilinear form, and $\IBF_R(J)$ is a free $R$-module admitting $\{\ta\}$ as a basis.
\end{theorem}

We leave it to the interested reader to look into the following possible improvement of this last result.

\begin{con} For $J$ as in Proposition~\ref{prop:Jcs}, is the $R$-module $\scIBF_R(J)$ projective?
\end{con}

Since $J$ is central and therefore a faithful $R$-module, we have $\ka_J(J,J) =R$ by \cite[I, Cor.~1.10]{DI}. Hence, if the question has a positive answer, Proposition~\ref{trans} applies and yields that $(J, \ta)$ satisfies the IBF-principle. In particular, Theorem~\ref{prop:Jcs} then becomes a corollary.


\section{Graded invariant bilinear forms} \label{sec:graded}

In this section we classify graded invariant bilinear forms, which are particularly important for infinite-dimensional Lie theory
. We will therefore concentrate on these (except for preliminaries considerations),  and leave the extension to other classes of algebras to the reader.
\sm

We begin with some generalities about gradings and graded forms. Unless specified otherwise, we continue with our standard setting: $k$ is a base ring and $B$ is an arbitrary $R$-algebra for some $R\in\kalg$. Throughout, $\La$ is an abelian group.

\begin{defn}[Graded algebras and graded invariant bilinear forms] A {\em $\La$-graded algebra\/} is a pair $(C,\scC)$ consisting of a $k$-algebra $C$ and a family $\scC= (C^\la)_{\la \in \La}$ of $k$-submodules $C^\la$ of $C$ satisfying $C= \bigoplus_{\la \in \La} C^\la$ and $C^\la C^\mu \subset C^{\la + \mu}$ for all $\la, \mu \in \La$. We will say that a $k$-algebra $C$ is \textit{$\La$-graded\/} if $(C, \scC)$ is a $\La$-graded algebra for some family $\scC$.  We point out that it is allowed that some of the homogeneous submodules $C^\la$ vanish. If $C$ is a unital algebra, then necessarily $1_C \in C^0$.

Assume $C$ is $\La$-graded. We call $\ka \in \scL^2_k(C)$ a {\em graded bilinear form\/} if $\ka(C^\la, C^\mu) = 0$ whenever $\la + \mu \ne 0$. \end{defn}

\begin{defn}[Graded $S/R$-forms of algebras]\label{def:grad-forms}
In the following we assume that $R\in \kalg$ is $\La$-graded, say $R=\bigoplus_{\la \in \La} R^\la$. An $R$-algebra $B$ is then called a \textit{$\La$-graded $R$-algebra\/} if $B= \bigoplus_{\la \in \La} B^\la$ is $\La$-graded as a $k$-algebra and the $\La$-gradings of $R$ and $B$ are compatible in the sense that $R^\la B^\mu \subset B^{\la + \mu}$ for all $\la, \mu \in \La$. For example, for any $k$-algebra $\fra$ the $R$-algebra $\fra \ot_k R$ is canonically a $\La$-graded $R$-algebra by defining the $\la$-homogeneous submodule $(\fra\ot_k R)^\la = \fra \ot_k R^\la$.

In the descent setting~\eqref{descent-setting} we suppose that $R\in \kalg$ is $\La$-graded and that $S\in \Ralg$ is a $\La$-graded $R$-algebra. We view $\fra\ot_k S$ with its canonical $\La$-grading. An $S/R$-form $B$ of $\fra \ot_k R$ is called \textit{graded} if $B$ is a $\La$-graded $R$-algebra and there exists an  $S$-algebra isomorphism $\theta \co B \ot_R S \to \fra \ot_k S$ which respects the gradings: $\theta(b^\la \ot s^\mu) \in \fra\ot_k S^{\la + \mu}$ for $b^\la \in B^\la$ and $s^\mu \in S^\mu$.
\end{defn}

We now specialize to Lie algebras and derive a graded version of Theorem~\ref{teo:Lie-desc}. Following the notation used in loc.~cit.\ we change $\fra$ to $\g$.

\begin{prop}\label{prop:grfo} Let $\g$ be a finite-dimensional semisimple Lie algebra over a field $k$ of characteristic $0$ with Killing form $\ka$. Assume that $R\in \kalg$ and $S\in \Ralg$ are $\La$-graded, $S/R$ is faithfully flat and $B$ is a graded $S/R$-form of $\g \ot_k R$.
\sm

{\rm (a)} If $\ka_B$ is the form attached to $\ka$ in Theorem~{\rm \ref{prop:referee}},  then $\ka_B$ is the Killing form of the $R$-algebra $B$ and  satisfies $\ka_B(B^\la, B^\mu) \subset R^{\la + \mu}$ for all $\la,\mu \in \La$.
\sm

{\rm (b)} If $\g$ is central (and hence simple), every graded invariant bilinear form $\be \in \IBF_k(B)$ can be written in the form $\vphi \circ \ka_B$ for a unique $\vphi \in \{ \vphi \in R^*: \vphi(R^\la) = 0\hbox{ for } \la \ne 0\} \simeq (R^0)^*$. \end{prop}

\begin{proof} (a) That $\ka_B$ is the Killing form of $B$, was established in Corollary~\ref{cor:bs}. Since there exists an $S$-algebra isomorphism $\theta \co B \ot_R S \to \g \ot_k S$ respecting the gradings, there is no harm to assume that $B \subset \g\ot_k S$ with $B^\la = B^\la \ot 1_S \subset \g \ot S^\la$. Recall from Corollary~\ref{cor:bs} that $\ka_B = \ka_{\g \ot_k S} \mid B \times B$ where $\ka_{\g \ot_k S}$ is the Killing form of the $S$-algebra $\g \ot_k S$, and that $\ka_{\g \ot_k S}$ coincides with the base change $\ka_S$ of $\ka$ by $S$. Because $\ka_S\big( \g\ot_k S^\la, \g\ot_k S^\mu)\subset S^{\la + \mu}$, it suffices to show that $R^\la=S^\la\cap R$ for all $\la \in \La$. To see this last statement, we use that $1_S\in S^0$, so that $R^\la = R^\la 1_S \in R^\la S^0 \subset S^\la$. Then $R^\la\subset S^\la\cap R$ follows. The other inclusion is immediate.

(b) We have seen in Theorem~\ref{teo:Lie-desc}(b) that $(B, \ka_B)$ satisfies the IBF-principle. It follows that $R=\Span_\ZZ \{ \ka_B(b_1, b_2) : b_i \in B\}$. Let now $\be \in \IBF_k(B)$ be a graded invariant bilinear form. Again by Theorem~\ref{teo:Lie-desc} there exists a unique $\vphi \in R^*$ such that $\be = \vphi \circ \ka_B$. We claim $\vphi(r) =0$ for any $r\in R^\la$, $\la \ne 0$. Indeed, there exist finitely many $b_i \in B^{\mu_i}$ and $b'_i \in B^{\la - \mu_i}$ such that $r = \sum_i \ka_B(b_i, b'_i)$. Hence $\vphi(r) = \sum_i \be(b_i, b'_i) = 0$. That, conversely, every $\vphi\in (R^0)^*$ gives rise to a graded invariant bilinear form, is of course obvious. \end{proof}

Proposition~\ref{prop:grfo} can be applied to multiloop algebras based on simple finite-dimensional Lie algebras. We recall their definition: $\g$ is a finite-dimensional simple Lie algebra over an algebraically closed field $k$ of characteristic $0$, and $\si = (\si_1, \ldots, \si_N)$ is a family of commuting automorphisms of $\g$, which have finite orders $m_1, \ldots, m_N$ respectively. We fix a set of primitive $m$-th roots of unity $\ze_m \in k$ which are compatible in the sense that $\ze_{ln}^l = \ze_n$, and put
\[  R = k[t_1^{\pm 1}, \ldots, t_N^{\pm 1}]\subset
   S = k[t_1^{\pm \frac{1}{m_1}}, \ldots, t_N^{\pm \frac{1}{m_N}}].
\]
The multiloop algebra $\caL = \caL(\g, \si)$ associated to these data is the Lie algebra
$$ \caL = \textstyle \bigoplus_{i_1, \ldots, i_N \in \ZZ^N} \, \g_{i_1, \ldots,i_N} \ot_k t_1^{\frac{i_1}{m_1}} \ldots t_N^{\frac{i_N}{m_N}} $$
where $\g_{i_1, \dots,i_N} = \{ x\in \g: \si_j(x) = \ze_{m_j}^{i_j} x \hbox{ for all } j\}$. Since $\g_{i_1, \dots,i_N}= \g_{i_1+k_1, \dots, i_N+k_N}$ for $(k_1 , \dots, k_N) \in m_1 \ZZ \oplus \cdots \oplus m_N \ZZ$, it is clear that $\caL$ is an $R$-Lie algebra. It is in fact an $S/R$-form of $\g \ot_k R$ (\cite[Th.~3.6]{abp2}).

To enter the grading into the picture,  we let $\Lambda = \frac{1}{m_1}\ZZ \times \cdots \times \frac{1}{m_N}\ZZ$. Then $\Lambda \simeq \ZZ^N$ and we have a natural $\Lambda$-grading on $S$ and $R$ (the reader will note that the homogeneous elements of $R$ have degrees in $\ZZ \times \cdots \times \ZZ \subset \Lambda$). The Lie algebra $\fg \ot_k S$ is naturally $\Lambda$-graded and this makes $\caL$ also naturally into a $\Lambda$-graded Lie algebra. It is immediate from the definitions that $\caL$ is $R$-graded and  a graded $S/R$-form of $\fg \ot_k S.$ Since in our situation $R^0 = k$, Proposition~\ref{prop:grfo} yields the first part of the following.

\begin{cor}\label{inv-ml} Let $\caL$ be a multiloop Lie algebra based on a simple finite-dimensional Lie algebra over an algebraically closed field $k$ if characteristic $0$. Then, up to scalars in $k$, the $\ZZ^N$-graded Lie algebra $\caL$ has a unique graded invariant $k$-bilinear form $\be$. It is given by
\[
     \be(x\ot t_1^{\frac{j_1}{m_1}} \ldots t_N^{\frac{j_N}{m_N}} , y \ot t_1^{\frac{l_1}{m_1}} \ldots t_N^{\frac{l_N}{m_N}} ) = \ka(x,y) \de_{j_1+l_1, 0} \ldots \de_{j_N+l_N, 0}
\]
where $\ka$ is the Killing form of $\fg$. The form $\beta$ is nondegenerate. Every $k$-linear automorphism of $\caL$ is orthogonal with respect to $\be$.
 \end{cor}

\begin{proof} The nondegeneracy of $\ka$ implies that $\be \mid \caL^\la \times \caL^{-\la}$ is a nondegenerate pairing for all $\la$, which in turn forces $\be$ to be nondegenerate.

Let now $f\in \Aut_k(\caL)$. The map $\chi \mapsto f \circ \chi \circ f^{-1}$ is an automorphism of the centroid of $\caL$. Since $\caL$ is central, $f$ is $\al$-semilinear for some $k$-linear automorphism $\al$ of $R$. By Proposition~\ref{kill-Lie}(b), we then know $\ka_\caL \circ (f \times f) = \al \circ \ka_\caL$ for $\ka_\caL$ the Killing form of the $R$-algebra $\caL$. The form $\be$ is obtained by composing $\ka_\caL$ with the canonical projection
$\epsilon \co R \to R^0=k$. It is therefore enough to show $\epsilon = \epsilon \circ \al$. But this is indeed the case: Every $k$-linear automorphism of $R$ fixes $R^0$ pointwise and permutes the $R^\la$, $\la \ne 0$. To see this, we realize $\Aut_k(R)$ as $\GL_N(\ZZ) \ltimes (k^\times)^N$ in the natural fashion. \end{proof}

This corollary is of interest for the construction of extended affine Lie algebras based on centreless Lie tori, which heavily depends on the existence of a graded invariant nondegenerate bilinear form on a Lie torus. The reader is referred to \cite{AABGP, n:eala,n:eala-summer,n:persp} for background material on extended affine Lie algebras and Lie tori. More precisely, we will deal here with Lie-$\ZZ^N$-tori.

\begin{cor}\label{cor:eala} Up to scalars, a centreless Lie-$\ZZ^N$-torus has a unique graded invariant nondegenerate bilinear form.
\end{cor}

\begin{proof} By \cite{n:tori} a centreless Lie-$\ZZ^N$-torus is either finitely generated over its centroid or is a Lie torus with a root system of type ${\rm A}$. Concerning the latter type, one knows from \cite{bgk,bgkn,y1} that they are graded isomorphic to $\mathfrak{sl}_n(k_q)$ for a quantum torus $k_q$, and for these types of Lie algebras the claim follows from \cite[7.10]{n:persp}. If $L$ is a centreless Lie-$\ZZ^N$-torus which is finitely generated over its centroid, then \cite[Th.~3.3.1]{ABFP2} says that $L$ is graded-isomorphic to a multiloop algebra so that we can apply the previous Corollary~\ref{inv-ml}.\end{proof}

Corollary~\ref{cor:eala} has been proven in \cite[Th.~7.1]{y:lie} for Lie tori graded by a torsion-free group $\La$, using the structure theory of these types of Lie tori.

\end{document}